\documentclass[10pt]{article}

\usepackage{color}
\usepackage[latin1]{inputenc}
\usepackage[english]{babel}
\usepackage{amsmath,amssymb,amsfonts,amsthm,enumerate}
\usepackage{mathrsfs,epsfig,graphicx,makeidx}

\usepackage{url}

\numberwithin{equation}{section} 

\newtheorem{theorem}{Theorem}[section]
\newtheorem{lemma}[theorem]{Lemma}
\newtheorem{definition}[theorem]{Definition}
\newtheorem{remark}[theorem]{Remark}
\newtheorem{proposition}[theorem]{Proposition}
\newtheorem{corollary}[theorem]{Corollary}

\newcommand{\N}{\mathbb{N}}
\newcommand{\R}{\mathbb{R}}

\renewcommand{\leq}{\leqslant}
\renewcommand{\le}{\leqslant}

\renewcommand{\ge}{\geqslant}

\title{Is a nonlocal diffusion
strategy convenient\\
for biological populations in competition?}
\author{Annalisa Massaccesi -- Enrico Valdinoci
\thanks{This work
has been supported by ERC grant 277749 ``EPSILON Elliptic
Pde's and Symmetry of Interfaces and Layers for Odd Nonlinearities''.}}
\begin{document}
\maketitle

\begin{center}
\begin{minipage}[c]{0.8\textwidth}\small
Abstract: {\sf We study the viability of a nonlocal dispersal strategy 
in a 
reaction-diffusion system with a fractional Laplacian operator. We show that there are 
circumstances - namely, a precise condition on the distribution of the resource - under which  the introduction of a new
nonlocal dispersal behavior is favored with respect to the local dispersal behavior of the resident population. 

In particular, we consider the linearization of a biological system that models
the interaction of two biological species, one with local and one with nonlocal
dispersal, that are competing for the same resource. We give a simple, concrete example
of resources for
which the equilibrium with only the local population
becomes linearly unstable. In a sense, this example shows that nonlocal strategies can invade an environment in which purely local
strategies are dominant at the beginning, provided that the resource
is sufficiently sparse.

Indeed, the example considered presents a high variance of the distribution of the dispersal,
thus suggesting that the shortage of resources
and their unbalanced supply may be some of the basic
environmental factors that favor nonlocal 
strategies.}
\end{minipage}
\end{center}
\bigskip\bigskip

MSC2010 Classification Numbers: 35Q92, 46N60.

\bigskip\bigskip

\begin{flushright}
{\footnotesize
{\it ``When the sun comes up, you better be running.''}

The Fable of the Lion and the Gazelle, 

popular quotation by Undetermined Author,

{\tt http://quoteinvestigator.com/2011/08/05/lion-gazelle/}

}
\end{flushright}
\bigskip

\section{Introduction}

The goal of this paper is to study the possible fate of
a nonlocal diffusion
strategy for a biological population
in presence of a highly oscillating distribution of resource.

The study of dispersal
strategies and the comparison between local and nonlocal
diffusive behaviors have recently attracted a great attention and several
researches have been developed both in terms of experiments
and from the purely mathematical point of view (see for instance 
\cite{Vis_al, Hum_al, Fried, Mon_Pel_Ver} and references therein).
Remarkably, the phenomenon of possibly nonlocal hunting strategies has attracted
also the attention of the mass-media, and related news can be found in popular
newspapers and magazines (see e.g. \cite{VEN}).

In this framework, even the distinction between local and nonlocal strategies is
somehow a delicate issue and it is still not exactly clear in all situations
what factors favor one behavior against the other.
Of course, in general, as we know even from 
experience in our everyday life, it may
be very difficult
to deduce from overall principles\footnote{By ``overall principles'' we mean the availability of a general method, depending on the measurement of some parameters in the environment, which allows a population to choose an optimal strategy. We are referring to the impossibility of having a satisfactory and complete model for population dynamics, due to the complexity of the biological world.}
the optimal strategy to follow in each complex situation.
Therefore, it is not surprising that
the question of detecting the optimal strategy
in a logistic mathematical model cannot have just a simple
answer that is valid in every situation, and, concretely,
very different dispersal
strategies have been directly
observed in nature.\medskip
 
Detecting, analyzing and understanding the differences
between diffusive
strategies is therefore a difficult, but
important, task in biology. One of the possible distinctions
among the different strategies lies in rigorously defining
the concept of ``locality''
(when a predator, roughly speaking, diffuses randomly
in the neighborhood looking for an available prey)
versus ``nonlocality'' (the short periods of hunting activity
are followed by rather long journeys of the predator in the
search for food). As expected, hunting strategies
of predators are definitely influenced by the distribution
of the resources.
When the resources are ``easily'' available,
it is conceivable that predators do not need to
elaborate a nonlocal hunting strategy and indeed it can be more
convenient not to drift too much to take advantage of
the rather abundant resource in their neighborhood. Conversely,
when the prey is sparse, it may be worth for
predators to interchange the local hunting activity
with suitable nonlocal travels in different possible regions.

Of course, the more sophisticated the species involved in the hunt, the easier
the latter phenomenon is expected to occur: namely,
an intelligent species of preys will 
run away from the danger, thus making the distribution of
resources for the predator sparse, and therefore making
a nonlocal hunting strategy possibly more favorable.
However, in the model considered in this paper the resource~$\sigma$ is
independent of the distribution of the populations, so this
effect is not taken into consideration by the setting discussed here.
\medskip

It is also evident that the distinction between local or nonlocal
strategy is a mathematical abstraction based on the consideration
of different space/time scales: i.e., the ambient space 
that the population has at its disposal
is not infinitely large in the real cases, and species cannot
really perform discontinuous, nonlocal jumps. Nevertheless,
a good mathematical model in which different scales are
taken into account may furnish a justification for
the diffusive strategy in a ``large enough'' environment
in which the time scales of travel and hunting activities
can be somehow distinguished in practice.
\medskip

We will try to give a rigorous mathematical framework to these 
na\"{\i}ves con\-si\-de\-ra\-tions by showing the possible advantages
of the long-jump dispersal strategies (i.e. the ones based
on nonlocal diffusion)
in regimes where the distribution of resources may be considerably
different at different points of the ambient space.
Not too surprisingly having in mind the concrete
applications, we will use for this scope
the mathematical framework of linearized systems
and scaling properties of the eigenvalues, which take
into account the stability property of equilibrium configurations.
\medskip

Our mathematical framework can be discussed as follows.
Reaction-diffusion systems provide
an effective continuous model for the biological problem of competition between different species.
The typical example of local
reaction-diffusion equation is 
\begin{equation}\label{cl_rd}
u_t=\Delta u+(\sigma-u)u\quad{\rm in}\ (0,T)\times\Omega\,.
\end{equation}
We study here the case of Dirichlet boundary conditions. Though other boundary conditions may be also taken into account to model different situations, our focus on the Dirichlet data is motivated by biological considerations (for instance, prescribing the solution to vanish outside a given domain corresponds to a confinement situation, for instance in a hostile environment).
In this model, the environment is represented by the open bounded set $\Omega\subset\R^n$, with~$n\ge2$, and a heterogeneous resource $\sigma:\Omega\to[0,+\infty)$ is given (stationary in time). The growth of the population density $u$ depends on a dispersal differential operator and on the reproductive rate of the population itself, which is proportional to the temporary availability of the resource $(\sigma-u)$. Dirichlet boundary conditions model a lethal environment for the population $u$ outside the domain $\Omega$. 

A reaction-diffusion system involves at least two species,
with distribution $u$ and~$v$, whose behavior is ruled
by a reaction-diffusion equation like \eqref{cl_rd}. 
The two competing species differ for some special features: 
indeed, \eqref{cl_rd} has to be modified in order to describe the 
foraging and reproductive habits of the species and further data 
concerning the environment. As it is customary in Adaptive Dynamics (see \cite{Diek} and \cite{Ha}), the first step in the study of the evolution of a given feature is to single it out and then assume that the two populations differ for this feature only. For instance, in our main\footnote{Up to Section \ref{sec:lin} we investigate the opposite situation, too, that is, when the resident population has a nonlocal dispersal strategy and the mutant population has a local one.} case the resident population has a local dispersal strategy and the mutant population has a nonlocal one. In \cite{Fried} and \cite{Diek}, one can find a 
comprehensive survey of the problem and of the standard approach in Adaptive Dynamics; many different features have been 
studied and compared in \cite{Doc_Hut_Mi_Per, Can_Co_Hu1, Can_Co_Hu2} (different dispersal rates and genetic mutations), in \cite{Hu_Mi_Po} (time-periodic sources) and in \cite{Can_Co_Lou1,Can_Co_Lou2,Chen_Ham_Lou} (addition of a chemotactic component depending on the gradient of the resource).

We are interested in the comparison of the dispersal strategies: in particular, we focus on the competition between a population with ``standard'' diffusion and a second population with nonlocal dispersal. Therefore, our model is
\begin{equation}\label{our_model}
\left\{
\begin{array}{l}
u_t=\hphantom{-(-)^s}\Delta u +\left(\sigma-(u+v)\right)u\\
v_t=-(-\Delta)^s v\,+\left(\sigma-(u+v)\right)v\,.\\
\end{array}
\right.
\end{equation}

At a discrete level, the ``standard'' assumption is that the motion of the po\-pu\-la\-tion is governed by a random walk and this obviously leads to a Laplacian operator in the continuous model. Analogously, since our interest is focused on a second population with nonlocal dispersal, we adopt the fractional Laplacian operator as dispersal operator for the 
second distribution. The choice of such nonlocal diffusion operator
is motivated by the fact that the fractional Laplacian has good stability properties in terms
of the associated stochastic processes (it is the ``continuous version'' of the discrete motion governed by L\'evy flights, see e.g.~\cite{Val1} for a simple motivation and~\cite{Be} for more advanced material), it possesses natural
scaling features and it seems also to appear in real experiments
(see e.g.~\cite{Vis_al, Hum_al}).
The present literature on the subject of nonlocal dispersal mostly
considers convolution operators (see \cite{Doc_Hut_Mi_Per,Ka_Lou_Shen1,Ka_Lou_Shen2,Can_Co_Lou_Ryan,Co_Da_Ma}).
In particular, in \cite{Ka_Lou_Shen1}, the model under investigation is
\begin{equation}\label{conv_model}
\left\{
\begin{array}{l}
u_t=\mu\Delta u +\left(\sigma-(u+v)\right)u\\
v_t=\nu\left(\delta^{-n}\int_D k\left(\frac{\cdot-y}{\delta}\right)v(y)\,dy-v\right)+\left(\sigma-(u+v)\right)v\,,\\
\end{array}
\right.
\end{equation}
where $\mu,\nu$ are the dispersal rates of the two populations, respectively, and $\delta$ is the dispersal distance of the second population.

Of course, it is a delicate business to decide, in concrete
situations, which models better describe the dispersion of
a real biological population, and many nonlocal terms have
been taken into account in order to comprise long-range effects.
In general, we believe that fractional equations may be an important
tool to further understand the complex problems arising
in the mathematical modelization of biological species
and we hope that the framework given in this paper can lead to
a further development of the subject.
\medskip

In Section \ref{sec:model} we provide
details and further explanations about the model considered here
and some basic facts about the fractional Laplacian operator.

We study the stability of a stationary solution $(\tilde u,0)$ of the aforementioned system, by means of a formal linearization at $(\tilde u,0)$, that we explain in Subsection~\ref{sec:lin}. The complete understanding
of the global dynamics of a general system of diffusive and competing populations
is beyond the scope of this paper and it seems, at first glance, 
very challenging from a mathematical point of view, since
a variety of possible situations may occur. Nevertheless, let us stress that even the analysis of the stability of a stationary solution (also called ``invasibility'' analysis) is
interesting and meaningful from an evolutionary point of view, as it is suggested in the principal literature in Adaptive Dynamics (again, see \cite{Diek}). In fact, a small perturbation around $(\tilde u,0)$ mirrors the occurrence of a genetic mutation in the first population, involving the dispersal strategy. At $(\tilde u,0)$ the first population benefits from an equilibrium state, while the second one does not even exist. Then a small portion of the first population (with density $\tilde u$) undergoes a genetic mutation, which starts a second population (with very small density $v$) which competes for the resource with the former. Of course, the genetic mutation of this theoretical experiment involves only the hunting/dispersal strategy, passing from a local to a nonlocal one. In this context, the expected outcome 
of the analysis of the stationary solution is, in most
of the cases experienced in practice, stability,
that is, the second population does not find the right conditions
to evolve
and it
gets rapidly extinguished. On the contrary, (even partial)
instability of these type of equilibria
is rather surprising and interesting, since in this case
the new dispersal strategy is convenient enough
to allow a short term survival of the second species and to provide a situation
of coexistence of two different populations.

\medskip

The core of this paper is Section \ref{sec:ciccia}, where we show how
the stability of $(\tilde u,0)$ (namely, the sign of the eigenvalues associated with the linearized system) depends on the distribution of the resource $\sigma$. In particular, we will show that
if a certain relationship between the variation of $\sigma$ and the 
fractional Poincar\'e-Sobolev constant in $\Omega$ is fulfilled (see Definition \ref{def:sscat}), then the linearized system has a positive eigenvalue and $(\tilde u,0)$ is unstable. It is transparent from Definition \ref{def:sscat} that the distributions leading to instability of $(\tilde u,0)$ (and suggesting convenience of a nonlocal dispersal strategy) are those with a ``huge variation''. The last part of Section \ref{sec:ciccia} is devoted to show that such a distribution $\sigma$ may occur. Summarizing, the 
result that states that the local dispersive strategy may become
unstable in presence of a new population endowed with nonlocal
diffusive strategies can be formally stated as follows:
 

\begin{theorem}\label{MAIN}
Let $\Omega\subset\R^n$ be an open subset of $\R^n$ with Lipschitz boundary and let $s\in (0,1)$. There exist bounded functions $\sigma:\Omega\to[0,+\infty)$ and~$\tilde u:\Omega\to[0,+\infty)$ such that $(u,v):=(\tilde u,0)$
is a linearly unstable equilibrium for the system
\begin{equation}\label{90}
\left\{
\begin{array}{rcll}
u_t= &\Delta u &\!\!\!\!+(\sigma-(u+v))u & \quad\text{in }\Omega\\
v_t= &\!\!\!\!-(-\Delta)^s v &\!\!\!\!+(\sigma-(u+v))v & \quad\text{in }\Omega\\
u\hphantom{_t}= & \!\!\!\!\!\!\!\!\!\!\!\!\!\!\!\!\!\!\!\!\!\!\!\! 0 & & \quad \text{on }\partial\Omega\\
v\hphantom{_t}= & \!\!\!\!\!\!\!\!\!\!\!\!\!\!\!\!\!\!\!\!\!\!\!\! 0 & & \quad \text{in }\R^n\setminus\Omega\,.
\end{array}
\right.
\end{equation}
More precisely, the function~$\tilde u$
is a solution of
\begin{equation}\label{90X}
\left\{
\begin{array}{ll}
\Delta\tilde u(x)+(\sigma(x)-\tilde u(x))\tilde u(x)=0 & {\mbox{ in }}\Omega,\\
\tilde u=0 & {\mbox{ on }}\partial\Omega,
\end{array}
\right.
\end{equation}
and the linearization of system in~\eqref{90} at~$(\tilde u,0)$
has a negative and a positive eigenvalue.
\end{theorem}

The existence of distributions of resource $\sigma$ which support the phenomenon described in Theorem \ref{MAIN} is motivated by concrete models. To be precise, in this paper we examine two particular cases: 
\begin{enumerate}
\item[1.] a rescaled resource $\sigma_\lambda$ on a sufficiently small domain $\Omega_\lambda$ in Section \ref{subsec:ra};
\item[2.] a large multiple of the characteristic function of a ball in Section \ref{subsec:cpt}.
\end{enumerate}

We remark that Theorem~\ref{MAIN} states that~$\tilde u$
is a linearly stable solution of the autonomous,
scalar Fisher-KPP equation in~\eqref{90X}, but~$(\tilde u,0)$
is a linearly unstable equilibrium for the system in~\eqref{90}.
More explicitly, the positive eigenvalue of the linearized system
takes into account the fact that if the density of the first
population undergoes a small variation without the appearance
of the second species, then the system has the tendency to return
to the original position. Conversely, the negative eigenvalue
shows that if a second population appears, then the system does
not go back to the original situation, and the second species
has indeed chances to survive and colonize the environment.

It is interesting to contrast this result with those obtained in \cite{Ka_Lou_Shen1}: when the dispersal rates are equal\footnote{In our model, we do not even take into account different dispersal rates $\mu$ and $\nu$.}, that is $\mu=\nu$, if the dispersal distance $\delta$ in \eqref{conv_model} is sufficiently small, then $v$ can invade the local population $u$ but $u$ cannot invade the nonlocal population $v$. This suggests, as a general principle, that the smaller spreader may be favored by evolution, especially in hostile environments. Our approach is rather different: in some sense, we consider the dispersal distance as already fixed (in the definition of fractional Laplacian) and we investigate the dependence of the possibility of an invasion based on the availability of the resource $\sigma$. Our attempt is to put in evidence the role of the environment (more precisely, the fact that Definition \ref{def:sscat} is fulfilled) in the selection of the dispersal strategy.

\medskip

Roughly speaking, the condition (in Definition \ref{def:sscat}) which 
allows the in\-sta\-bi\-li\-ty of the system records the fact that the 
first population, with local diffusion, cannot saturate the given 
resource and leaves enough ``leftovers'' for the second species to 
survive. 
\medskip

In this sense, a natural question is to
determine whether a population exhausts the resource.
For this, as a second result,
we provide an example
of a purely nonlocal phenomenon in population modeling. We
show that, fixed any
arbitrarily small~$\varepsilon>0$ and
given any resource~$\sigma\in
C^k(B_1,\,[0,+\infty))$, there exists a 
resource~$\sigma_\varepsilon\in
C^k(B_1,\,[0,+\infty))$ that is~$\varepsilon$-close to~$\sigma$
in the norm of~$C^k(B_1)$, a radius~$R_{\varepsilon,\sigma}>1$
and a function~$u_\varepsilon$ which vanishes outside~$B_{R_{\varepsilon,\sigma}}$,
which is~$s$-harmonic in~$B_1$,
which equals to~$\sigma_\varepsilon$ in~$B_1$ and which therefore
satisfies
$$ (-\Delta)^s u_\varepsilon =(\sigma_\varepsilon - u_\varepsilon)u_\varepsilon
\quad{\mbox{ in~$B_1$.}}$$
That is, up to an arbitrarily small error,
a nonlocal population can locally adapt to any given resource
(provided that the density of the population
is artificially and appropriately regulated in a suitable region).
The formal statement of this result goes as follows.

\begin{theorem}\label{NNC}
Let~$k\in\N$ and~$\sigma\in C^k(B_1,\,[0,+\infty))$.
Fix~$\varepsilon>0$.
Then there exists $\sigma_\varepsilon \in
C^k(B_1)$ with  
\begin{equation}\label{CU1}
\|\sigma-\sigma_\varepsilon\|_{C^k(B_1)}\le\varepsilon
\end{equation}
and there exist $R_{\varepsilon,\sigma}>1$ and~$u_\varepsilon \in C^k(B_1)\cap C^s(\R^n)$ such that
\begin{eqnarray}
&& u_\varepsilon (x) =\sigma_\varepsilon (x)\quad \text{for any }x\in B_1\label{CU4}\\
&& (-\Delta)^s u_\varepsilon(x)  =0 \quad \text{for any }x\in B_1\label{CU2}\\
&& u_\varepsilon (x) =0  \quad \text{for any }x\in \R^n\setminus B_{R_{\varepsilon,\sigma}}\,.\label{CU3}
\end{eqnarray}
In particular
\begin{equation}\label{CU5}
(-\Delta)^s u_\varepsilon(x) =(\sigma_\varepsilon(x) - u_\varepsilon(x))\,u_\varepsilon(x)
\quad \text{for any }x\in B_1\,.
\end{equation}
\end{theorem}

It is worth mentioning that Theorem~\ref{NNC}
heavily relies on the nonlocal feature of the equation and
it does not have any local counterpart
(this will be clearly explained in Section~\ref{sec:cur}). Let us stress the fact that Theorem~\ref{NNC}
does not prove (and cannot prove,
since this would be false in general)
that a nonlocal population always exhausts completely
the resource, since a small error~$\varepsilon$ has
to be taken into account. In a sense,
the solution given by Theorem~\ref{NNC} is different
than the original one, since it does not attain the
homogeneous Dirichlet boundary datum: it has prescribed,
non-homogeneous (but compactly supported)
Dirichlet boundary datum outside the strategic region in which
the equation is satisfied\footnote{The fact that the Dirichlet boundary condition is not homogeneous reflects mathematically the practical condition of performing an effective distribution plan for the population outside the strategic region.}.

We observe that Theorem~\ref{NNC}
has important (though socially
not embraceable!) practical consequences.
For instance, a given population may have a strong intention to consume
all the given resource in a region of particular strategic importance
(say, a region contained in
the ball~$B_1$ in our example). Indeed, in concrete cases,
this strategic area might be favorable for generating a new competing
species, or might be easily accessible
by a similar population coming from abroad which can be
considered dangerous or undesired
by the local population,
and the possible leftover of the resource might obviously favor the
newcomers. For these reasons, a ``socially conservative''
(and rather unkind!)
population may wish to avoid to leave available resources
in strategic regions which can be used by unwanted competitors.

The result in Theorem~\ref{NNC} says,
roughly speaking, that in this case, a nonlocal population is able to
find a suitable, somehow ``artificial'', distribution of population far
away in order to consume the resource in the strategic region and thus
penalize the newcomers (viceversa, a local population cannot do that).

Notice that this suitable distribution of
the conservative population may require a modification of the conditions
far away: indeed the supporting ball $B_{R_{\varepsilon,\sigma}}$ in
Theorem~\ref{NNC}
may become larger and larger for small $\varepsilon$:
that is, in a sense, the conservative population may need to change its
plan ``close to infinity'' in order to consume more efficiently the
inner resource (in this sense, a ``global plan'' for the
population distribution is in order, and it is indeed
conceivable that an optimal use of resources may involve
strategic plans on the distribution of the population
in the large).

\medskip

The rest of this paper is organized as follows.
In Section~\ref{sec:model} we recall the basic notation
about the population dynamics model that we study.
The linearized dynamics of the system is then analyzed in
Section~\ref{sec:ciccia}, where we will also give
two examples that establish Theorem~\ref{MAIN}.
Finally, in Section~\ref{sec:cur} we will prove
Theorem~\ref{NNC} and show that it is a new phenomenon,
which only arises in nonlocal dispersion models.

\section{Biological models and mathematical tools}\label{sec:model}
\subsection{Population dynamics}

Let us denote by $u,v:[0,T)\times\Omega\to[0,+\infty)$ the densities of two species coexisting in the same domain $\Omega$ and competing for a common resource $\sigma:\Omega\to\R$. Here and in the rest of the paper we consider as a domain an open, bounded set $\Omega\subset\R^n$ with Lipschitz boundary $\partial\Omega$.
The resource $\sigma$ belongs to the space of measurable, essentially bounded functions $L^\infty(\Omega)$. 
We study the linear
stability of a stationary solution of the reaction-diffusion system with 
Dirichlet boundary conditions
\begin{equation}\label{syst_dyn}
\left\{
\begin{array}{ll}
u_t=\hphantom{-(-)^s}\Delta u +\left(\sigma-(u+v)\right)u\quad &\text{in }[0,T)\times\Omega\\
v_t=-(-\Delta)^s v\,+\left(\sigma-(u+v)\right)v\quad &\text{in }[0,T)\times\Omega\\
u(t,\cdot)=0 & \text{on }\partial\Omega,\,\forall\,t\in[0,T)\\
v(t,\cdot)=0 & \text{in }\R^n\setminus\Omega,\,\forall\,t\in[0,T)\,.
\end{array}
\right.
\end{equation}
For this, we perform a formal linearization around a stationary point $(\tilde u,0)$ of \eqref{syst_dyn} and then we focus only on the corresponding
linearized system, that is
\[
\left\{
\begin{array}{ll}
\hphantom{(}-\Delta u\hphantom{)^s}\!\!\!=(\sigma-2\tilde u)u-\tilde uv \quad & {\rm in}\ \Omega\\
(-\Delta)^sv=(\sigma-\tilde u)v \quad & {\rm in}\ \Omega\\
\hphantom{(}u=0 \quad & {\rm on}\ \partial\Omega\\
\hphantom{(}v=0 \quad & {\rm in}\ \R^n\setminus\Omega\,.
\end{array}
\right.
\]
\begin{remark}{\rm
Though the global dynamics is beyond the scope of this paper, we recall that there is a detailed and specialized
literature about the well-posedness of the initial value problem associated to \eqref{syst_dyn} with $u(0,x)=u_0(x)$ and $v(0,x)=v_0(x)$ for some given functions $u_0,v_0$ (see \cite{Mo} for the well-posedness of the problem in $C^k(\overline\Omega)$, for instance). The analysis of the global dynamics of semilinear parabolic systems is performed through the theory of Monotone Dynamical Systems (see \cite{Hirsch}, \cite{Hirsch2} and \cite{Smith}). For the fractional Fisher-KPP equation one should see \cite{Ca_Ro1} and \cite{Be_Ro_Ro}, for instance. See also Section \ref{sec:glob} for further remarks.
}\end{remark}
Before focusing on the aforementioned linearized system, let us recall some useful definitions and facts about the pseudodifferential operator $(-\Delta)^s$ that is involved in \eqref{syst_dyn}.


\subsection{The nonlocal dispersive strategy and the fractional Laplacian}\label{subsec:fl}

Consider an open set $\Omega\subset\R^n$ and $s\in(0,1)$, the Gagliardo seminorm of a measurable function $u$ is defined as
\[
[u]_{H^s(\R^n)}:=\left(\int_{\R^n} \int_{\R^n}
\frac{|u(x)-u(y)|^2}{|x-y|^{n+2s}}\,dx\,dy\right)^{\frac 12}\,.
\]
The fractional Sobolev space that we denote here~$H^s_0(\Omega)$
is the linear set containing all the measurable functions~$u:\R^n\to\R$
such that:
\begin{itemize}
\item $\|u\|_{L^2(\Omega)}<+\infty$,
\item $[u]_{H^s(\R^n)}<+\infty$, and
\item $u(x)=0$ for a.e.~$x\in\R^n\setminus\Omega$.
\end{itemize}
The Gagliardo seminorm is naturally related to the fractional Laplacian,
since
\[
(-\Delta)^su(x):=\left(\frac{\Gamma(n/2+s)}{\pi^{2s+n/2}\Gamma(-s)}\right)\lim_{\varepsilon\to 0}\int_{\R^n\setminus B_\varepsilon(x)}\frac{u(x)-u(y)}{|x-y|^{n+2s}}\,dy\,,
\]
where~$\Gamma$ is the Euler's function.
For an introduction to the fractional Laplacian and the fractional
Sobolev spaces see for instance~\cite{Di_Nez_Pa_Val}. 
In our framework,
the scalar version of \eqref{syst_dyn}, that is
\[
v_t=-(-\Delta)^sv+(\sigma-v)v\,,
\] 
is known as Fisher-KPP equation with fractional diffusion and for the many established results one can see, for instance, \cite{Ca_Ro1} and \cite{Stan_Vaz}.

In this section we summarize the results needed in this paper only. 
\begin{theorem}[Fractional Poincar\'e-Sobolev embedding theorem]\label{thm_sobineq}
Fix $s\in(0,1)$ 
and an open bounded set $\Omega\subset\R^n$ with Lipschitz boundary. There exists a positive constant $C_\sharp=C_\sharp(s,\Omega)$ such that
\begin{equation}\label{fsi}
\forall\,\phi\in H^s_0(\Omega)\,,\quad\|\phi\|^2_{L^2(\Omega)}\le C_\sharp[\phi]^2_{H^s(\R^n)}\,.
\end{equation}
This means that $H^s_0(\Omega)$ is continuously embedded in $L^2(\Omega)$.
\end{theorem}

\begin{proof}
We give the proof, which is of classical flavor, for the facility
of the reader. We argue by contradiction, supposing that
there exists a sequence~$\phi_k\in H^s_0(\Omega)$
such that~$\|\phi_k\|_{L^2(\Omega)}\ge k[\phi_k]_{H^s(\R^n)}$.
We define
$$ \psi_k := \frac{\phi_k}{\|\phi_k\|_{L^2(\Omega)} }.$$
Then~$\psi_k\in H^s_0(\Omega)$ and
\begin{equation}\label{p0o1}
\left( \int_{\R^n}\int_{\R^n} \frac{|\psi_k(x)-\psi_k(y)|^2}{
|x-y|^{n+2s}}\,dx\,dy\right)^{\frac12}=
[\psi_k]_{H^s(\R^n)} =
\frac{ [\phi_k]_{H^s(\R^n)} }{ \|\phi_k\|_{L^2(\Omega)} }\le
\frac{1}{k}.\end{equation}
Therefore
$$ \left(\int_{\Omega}\int_{\Omega} \frac{|\psi_k(x)-\psi_k(y)|^2}{
|x-y|^{n+2s}}\,dx\,dy\right)^{\frac12}\le \frac{1}{k}.$$
Also, $\|\psi_k\|_{L^2(\Omega)}=1$.
Therefore, by compactness (see e.g. Theorem~7.1
in~\cite{Di_Nez_Pa_Val}, used here with~$p=q=2$), we obtain that,
up to a subsequence, $\psi_k$ converges to some~$\psi$
in~$L^2(\Omega)$ and a.e. in~$\Omega$. Defining~$\psi(x):=0$
for any~$x\in\R^n\setminus\Omega$, we have that~$\psi_k=\psi=0$
a.e. in~$\R^n\setminus\Omega$, and consequently~$\psi_k$ converges
to~$\psi$ a.e. in~$\R^n$.

Thus, by taking the limit in~\eqref{p0o1}
and using Fatou's Lemma,
\begin{eqnarray*}&&\int_{\R^n}\int_{\R^n} \frac{|\psi(x)-\psi(y)|^2}{
|x-y|^{n+2s}}\,
dx\,dy\le
\liminf_{k\to+\infty}
\int_{\R^n}\int_{\R^n} \frac{|\psi_k(x)-\psi_k(y)|^2}{
|x-y|^{n+2s}}\,
dx\,dy\\&&\qquad\le \liminf_{k\to+\infty}\frac{1}{k^2}=0.\end{eqnarray*}
Accordingly~$\psi$ must be constant in~$\R^n$ and therefore
identically equal to zero (up to sets of null measure).
This implies that
$$ 1= \lim_{k\to+\infty} \|\psi_k\|_{L^2(\Omega)}
=\lim_{k\to+\infty} \|\psi_k-\psi\|_{L^2(\Omega)}=0.$$
This is a contradiction and it proves the desired result.
\end{proof}

In the following, we will always assume $C_\sharp(s,\Omega)$ to be the sharp constant such that \eqref{fsi} holds, namely
\begin{equation}\label{CP}
C_\sharp^{-1}(s,\Omega) =
\inf_{\substack{\phi\in H^s_0(\Omega) \\ \phi\not\equiv0}}
\frac{[\phi]^2_{H^s(\R^n)}}{\|\phi\|^2_{L^{2}(\Omega)} } =
\inf_{\substack{\phi\in H^s_0(\Omega) \\ \phi\not\equiv0}}
\frac{[\phi]^2_{H^s(\R^n)}}{\|\phi\|^2_{L^{2}(\R^n)} }\,.
\end{equation}

\begin{remark}\label{rmk:rs}{\rm
If~$r>0$ and~$\phi\in H^s_0(B_1)$, one can consider
the rescaled function~$\phi_r(x):=r^{-n/2}\phi(x/r)$.
Then~$\phi_r$ vanishes a.e. outside~$B_r$. Moreover,
$\|\phi_r\|_{L^{2}(\R^n)}=\|\phi\|_{L^{2}(\R^n)}$
and~$ [\phi_r]_{H^s(\R^n)}=r^{-s} [\phi]_{H^s(\R^n)}$.
Accordingly,
\[
C_\sharp(s,B_r)=r^{2s} C_\sharp(s,B_1)\,.
\] 
}\end{remark}


\subsection{Linearization of the system}\label{sec:lin}

Let $\Omega\subset\R^n$ and $\sigma\in L^\infty(\Omega)$ be as in Section \ref{sec:model}. Our purpose is a qualitative study of an equilibrium state of the following system
\begin{equation}\label{syst}
\left\{
\begin{array}{l}
u_t=\hphantom{-(-}\Delta u\hphantom{)^s}\!+\left(\sigma-(u+v)\right)u\\
v_t=-(-\Delta)^sv+\left(\sigma-(u+v)\right)v
\end{array}
\right.
\end{equation}
More precisely, we look for an equilibrium state of the form $(\tilde u,0)$ with $\tilde u\in H^1_0(\Omega)$ and $\tilde u\ge 0$.

For the sake of completeness, we also investigate the existence of an equilibrium state of the form $(0,\tilde v)$, with $\tilde v\in H^s_0(\Omega)$ and $\tilde v\ge 0$. The linearization of \eqref{syst} at $(0,\tilde v)$ and further conclusions are postponed to Section \ref{sec:cur}.

\begin{definition}\label{def:nontriv}{\rm
Given a bounded function~$\sigma:\Omega\to[0+\infty)$,
we say that~$\sigma$ satisfies a 
reverse Poincar\'e-Sobolev condition if
\begin{equation}\label{poscond}
\sup_{u\in H^1_0(\Omega)}
\int_\Omega\sigma(x)u(x)^2\,dx-\int_\Omega|\nabla u|^2\,dx>0\,.
\end{equation}
Furthermore, $\sigma$ satisfies a reverse fractional Poincar\'e-Sobolev condition with parameter $s$ if
\begin{equation}\label{poscond_frac}
\sup_{v\in H^s_0(\Omega)}\int_\Omega\sigma(x)v(x)^2\,dx-[v]^2_{H^s(\R^n)}>0\,.
\end{equation}
}\end{definition}

In order to make computations easier, we give a sufficient
condition that ensures~\eqref{poscond}.

\begin{lemma}\label{COMP}
Let~$\lambda_1(\Omega)$ be the first eigenvalue of
the Laplacian in~$\Omega$ with Dirichlet boundary
condition and let~$\phi_1\in H^1_0(\Omega)$ be the
corresponding eigenfunction.
If
\begin{equation}\label{hp_exist} \lambda_1(\Omega)\int_\Omega
\phi_1(x)^2\,dx<\int_\Omega\sigma(x)\phi_1(x)^2\,dx\,, \end{equation}
then the reverse Poincar\'e-Sobolev condition in~\eqref{poscond} is satisfied.
\end{lemma}

\begin{proof} By construction
\begin{equation*} \left\{ \begin{array}{ll}
-\Delta\phi_1=\lambda_1(\Omega)\phi_1\quad & \text{in }\Omega\\ \phi_1=0
& \text{on }\partial\Omega\,, \end{array} \right.\end{equation*}
and so, by~\eqref{hp_exist},
$$ \int_\Omega |\nabla\phi_1(x)|^2\,dx=
\lambda_1(\Omega)\int_\Omega
\phi_1(x)^2\,dx<\int_\Omega\sigma(x)\phi_1(x)^2\,dx$$
which proves~\eqref{poscond}.\end{proof}

\begin{remark}\label{0oUU} {\rm It is worth noticing that
condition~\eqref{hp_exist}
is satisfied, for a fixed domain~$\Omega$,
for any resource~$\sigma$ that is sufficiently large
in an open subset of~$\Omega$.
Hence, fixed~$\Omega$, there are many examples
of smooth resources satisfying~\eqref{hp_exist}
and therefore~\eqref{poscond}.}
\end{remark}

\begin{remark} {\rm 
We also observe that the converse of Lemma~\ref{COMP} does not hold true, 
i.e. the reverse Poincar\'e-Sobolev condition in~\eqref{poscond} does not 
necessarily imply~\eqref{hp_exist}: as an example, one may consider 
$\Omega=(0,\pi)$, $\sigma(x) = \varepsilon^{-29/10}\chi_{ (0,\varepsilon)}(x)$ 
and~$u(x)=|x|^{2/3}$, with~$\varepsilon>0$ suitably small. Then~$u\in 
H^1_0(\Omega)$ and~\eqref{poscond} holds true, since 
\begin{eqnarray*}
&& \int_\Omega\sigma(x)u(x)^2\,dx-\int_\Omega|\nabla u|^2\,dx = 
\varepsilon^{-29/10} \int_0^{\varepsilon}x^{4/3}\,dx -\frac49\int_0^\pi
x^{-2/3}\,dx\\&&\qquad=\frac{3}{7} \varepsilon^{-29/10} \varepsilon^{7/3} 
-\frac{4\pi^{1/3}}{3}>0.\end{eqnarray*} On the other hand,
in this case~$\phi_1(x)=\sin x$, $\lambda_1(\Omega)=1$,
and
\begin{align*} 
\hphantom{=} & \,\lambda_1(\Omega)\int_\Omega \phi_1(x)^2\,dx-
\int_\Omega\sigma(x)\phi_1(x)^2\,dx\\
= & \,\int_0^\pi \sin^2 x\,dx -\varepsilon^{-29/10} \int_0^\varepsilon 
\sin^2 x\,dx \\
= & \,\frac{\pi}{2}
-\frac{\varepsilon^{-29/10}}{2}\big( \varepsilon-\sin\varepsilon\cos\varepsilon\big)\,.
\end{align*}
Thus, since, by a Taylor expansion,
$$ \sin\varepsilon\cos\varepsilon = (\varepsilon +O(\varepsilon^3))(1+O(\varepsilon^2))
= \varepsilon +O(\varepsilon^3)$$
it follows that
$$ \varepsilon^{-29/10} \big( \varepsilon-\sin\varepsilon\cos\varepsilon\big)
= O(\varepsilon^{1/10}) $$
and so
$$ \lambda_1(\Omega)\int_\Omega \phi_1(x)^2\,dx-
\int_\Omega\sigma(x)\phi_1(x)^2\,dx =\frac{\pi}{2}-O(\varepsilon^{1/10})
>0,$$
which shows that~$\phi_1$
does not satisfy~\eqref{hp_exist}.
}\end{remark}

The fractional equivalent of Lemma \ref{COMP} is stated in the following lemma. We omit its proof, which would be a repetition of the proof of Lemma \ref{COMP}.  
\begin{lemma}\label{lem:comp_frac}
Let $\phi_s\in H^s_0(\Omega)$ a minimizer for the Rayleigh quotient \eqref{CP}. If 
\[
C_\sharp^{-1}(s,\Omega)\|\phi_s\|^2_{L^2(\Omega)}<\int_\Omega\sigma(x)\phi_s(x)^2\,dx\,,
\]
then the reverse fractional Poincar\'e-Sobolev condition with parameter $s$ in \eqref{poscond_frac} is satisfied.
\end{lemma}

\begin{remark}\label{rmk:mangime}{\rm As we noticed in Remark \ref{0oUU} for the local analog of the reverse fractional Poincar\'e-Sobolev condition, if the resource $\sigma$ is sufficiently abundant in the domain $\Omega$, then condition \eqref{poscond_frac} is satisfied. }
\end{remark}

The reverse Poincar\'e-Sobolev condition in~\eqref{poscond}
is a useful tool to obtain
non-trivial solution of the local stationary equation,
as stated in the following result.

\begin{theorem}\label{thm:sta_sol}
Consider $s\in(0,1]$ and a bounded function $\sigma:\Omega\to[0,+\infty)$
satisfying either
\begin{itemize}
\item[{\rm (1)}] the reverse Poincar\'e-Sobolev condition in \eqref{poscond} (when $s=1$) or
\item[{\rm (2)}] the reverse fractional Poincar\'e-Sobolev condition in \eqref{poscond_frac} (when $s<1$).
\end{itemize}
Then there exists a non-trivial, non-negative function $\tilde u\in H^s_0(\Omega)$ (i.e.~$\tilde u\ge 0$ and $\tilde u\not\equiv 0$) satisfying
\begin{equation}\label{eig0}
\left\{
\begin{array}{ll}
(-\Delta)^s\tilde u(x)=(\sigma(x)-\tilde u(x))\tilde u(x)\quad & \text{in }\Omega\\
\tilde u=0 & \text{on }\partial\Omega\,.
\end{array}
\right.
\end{equation}
\end{theorem}
\begin{proof}
The proof is a minimization argument, based on coercivity and energy methods. Though the idea of the proof is rather standard,
see e.g. \cite{Can_Co}, we provide the necessary details for the facility of the reader. The proof also gives us the possibility of a comparison between local and nonlocal case.

First, we prove the theorem for the local case $s=1$ and then we provide the suitable changes in order to prove the nonlocal case, too.
\begin{itemize}
\item[(1)] Consider the following energy
\[
E(u):=\int_\Omega \frac{|\nabla u|^2}{2}-\sigma\frac{u^2}{2}+\frac{|u|^3}{3}
\]
defined on $H^1_0(\Omega)$. Notice that the Euler-Lagrange equation for $E$ gives
\[
-\Delta u=(\sigma-|u|)u\,.
\]
We show that the energy $E$ is coercive in $H^1_0(\Omega)$, that is
\begin{equation}\label{COE}
E(u)\to +\infty\ \text{as }\|u\|_{H^1_0(\Omega)}\to+\infty\,.
\end{equation}
For this, we use the Young inequality with exponents~$3/2$ and~$3$
to see that, for any~$a$, $b\ge0$,
\[
ab\le \frac 23 a^{\frac 32}+\frac 13 b^3\,.
\]
In particular, taking~$a:=2^{-2/3} u^2 $ and~$b:=2^{-1/3}
\|\sigma\|_{L^\infty(\Omega)}$, we obtain that
$$ \sigma\frac{u^2}{2}
\le \|\sigma\|_{L^\infty(\Omega)}\frac{u^2}{2}
\le \frac{|u|^3}{3}+\frac{
\|\sigma\|_{L^\infty(\Omega)}^{3} }{6},$$
hence
$$ -\sigma\frac{u^2}{2} +\frac{|u|^3}{3} \ge -c_0,$$
for some~$c_0>0$ independent of~$u$. Accordingly,
$$ E(u)\ge \int_\Omega \frac{|\nabla
u(x)|^2}{2}\,dx-c_0|\Omega|,$$
that establishes~\eqref{COE}.

As a consequence of~\eqref{COE},
we have that~$E$ has a global minimum $\overline u\in H^1_0(\Omega)$, satisfying
\[
-\Delta\overline u=(\sigma-|\overline u|)\overline u\,.
\]
Since $\overline u$ is a minimum, then $\tilde u:=|\overline u|$ is a minimum too, because $E(\overline u)=E(|\overline u|)$. Thus we can consider a non-negative function $\tilde u\ge 0$ satisfying
\[
-\Delta\tilde u=(\sigma-\tilde u)\tilde u\,.
\]

We conclude the proof by
showing that condition \eqref{poscond} guarantees that $E(\tilde u)<0$
and then $\tilde u\not\equiv 0$. By~\eqref{poscond},
there exists a function $u\in H^1_0(\Omega)$ with
\[\int_\Omega\sigma(x)u(x)^2\,dx-\int_\Omega|\nabla u|^2\,dx>0\,.\]
By density, we can suppose that~$u\in C^\infty_0(\Omega)$.
For every $\varepsilon>0$ we can rewrite the energy $E$ evaluated at $\varepsilon u$ as
\[
E(\varepsilon u)=\varepsilon^2\left(\int_\Omega\frac{|\nabla u|^2}{2}-\sigma\frac{u^2}{2}+\varepsilon\frac{u^3}{3}\right)\,,
\] 
hence $E(\tilde u)\le E(\varepsilon u)<0$ provided $\varepsilon$ is small enough.
\item[(2)] The energy
\[
E_s(v):=\int_{\R^n\times\R^n}\frac{|v(x)-v(y)|^2}{|x-y|^{n+2s}}\,dx\,dy-\int_\Omega\sigma\frac{v^2}{2}+\frac{|v|^3}{3}
\]
is well defined\footnote{For the sake of simplicity, we omit the multiplicative normalization constants.} and coercive in $H^s_0(\Omega)$ and the proof is the same as in the local case. Moreover, the Euler-Lagrange equation for $E_s$ is
\[
(-\Delta)^sv=(\sigma-|v|)v\,.
\]
Consequently, $E_s$ has a global minimum $\overline v$ and $\tilde v=|\overline v|$ is a minimum, too, because
\begin{align*}
E_s(\tilde v) & =\int_{\R^n\times\R^n}\frac{||\overline v(x)|-|\overline v(y)||^2}{|x-y|^{n+2s}}\,dx\,dy-\int_\Omega\sigma\frac{|\overline v|^2}{2}+\frac{|\overline v|^3}{3}\\
& \le \int_{\R^n\times\R^n}\frac{|\overline v(x)-\overline v(y)|^2}{|x-y|^{n+2s}}\,dx\,dy-\int_\Omega\sigma\frac{\overline v^2}{2}+\frac{|\overline v|^3}{3}=E_s(\overline v)\,.
\end{align*}
As we proved in part (1), condition \eqref{poscond_frac} ensures that $E_s(\tilde v)<0$ and thus $\tilde v\not\equiv 0$.\qedhere
\end{itemize}
\end{proof}

The result in Theorem~\ref{thm:sta_sol} and several variations of it
are rather of classical flavor: with slightly different assumptions on $\sigma$ (take, for instance, $\sigma>0$ in $\Omega$) and a branching condition matching \eqref{poscond} for the existence of non-trivial solutions, it can be found in \cite{Am_Pro} and in \cite{Be}.

In view of Theorem \ref{thm:sta_sol} and comparing with Remark \ref{0oUU} and Remark \ref{rmk:mangime}, we obtain the fact that the richer the environment is, the easier the survival of a population. This fact, which matches the intuition, finds a detailed quantification in the following observation.

\begin{remark}{\rm
Since the reverse (fractional) Poincar\'e-Sobolev inequalities \eqref{poscond} and \eqref{poscond_frac} seem to play a symmetric role in Theorem \ref{thm:sta_sol}, let us compare them more carefully. In some sense, the Dirichlet boundary conditions being equal for \eqref{eig0} when $s=1$ and when $s<1$, the nonlocal population has an advantage when the diameter of the domain tends to $0$. 

More precisely, as we remarked in Remark \ref{0oUU}, a resource $\sigma$ needs to be sufficiently large in order to meet \eqref{hp_exist}, which implies \eqref{poscond}. How large should $\sigma$ be is proportional to the first eigenvalue of the Laplacian $\lambda_1(\Omega)=C_\sharp(1,\Omega)^{-1}$. Now, if $\Omega=B_r$, we observe that
\[
\lambda_1(B_r)=\frac{1}{r^2C_\sharp(1,B_1)}\longrightarrow +\infty\qquad\text{as }r\to 0\,, 
\]
that is, the environment becomes more and more lethal for the local population, because \eqref{hp_exist} is very difficult to satisfy. The situation is milder for the nonlocal population, because, thanks to Remark \ref{rmk:rs}
\[
\frac{1}{\lambda_1(B_r)C_\sharp(s,B_r)}=\frac{r^2C_\sharp(1,B_1)}{r^{2s}C_\sharp(s,B_1)}\longrightarrow 0\qquad\text{as }r\to 0\,.
\]
This means that the criticality of the domain size is slower to prevail on a nonlocal population. 

In this sense, the lethal property of the boundary
(as described by the homogeneous Dirichlet datum
outside the domain) has a different
influence on local and nonlocal populations, depending
on the scale of the domain. For instance,
for small balls (when~$1/r\gg 1/r^s$),
nonlocal populations are favored. Conversely,
for large balls (when~$1/r\ll 1/r^s$),
local populations are favored (heuristically,
because the local diffusion has little chance to
reach the deadly boundary).

In the perspective of an applied analysis, one can find explicit fractional Sobolev constants in \cite{Co_Ta}.
}\end{remark}

In the remaining part of this section, we focus on the local case, that is, on the properties of a stationary point for the system \eqref{syst} of type $(\tilde u,0)$. This is motivated by the evolutionary point of view of studying the effect of the advent of a new population (this approach is indeed often adopted in the literature, see e.g. \cite{Doc_Hut_Mi_Per}, \cite{Ha}). Of course, we think that it would be also an interesting problem to investigate the cases of a dominant nonlocal population (corresponding to a stationary point of type $(0,\tilde v)$) and of the possible coexistence of two different populations, namely provide concrete assumptions on the resources and the domains that allow the existence of equilibria~$(u^*,v^*)$ with both~$u^*$ and~$v^*$ nontrivial. 

Of course, an easier approach to the existence of mixed states~$(u^*,v^*)$
may be taken by
studying the case of different resources in the two logistic equations, but we do not address this problem
in the present paper.

\begin{remark}{\rm
As a byproduct of the proof of
Theorem~\ref{thm:sta_sol}, we have that the solution
found is an energy minimizer. That is, if~$\tilde u$
is the solution obtained in Theorem~\ref{thm:sta_sol},
then~$E(\tilde u+\varepsilon u)\ge E(\tilde u)$,
for any~$u\in H^1_0(\Omega)$. Accordingly, the map
$$ \varepsilon \mapsto {\mathcal{E}}(\varepsilon):=E(\tilde u+\varepsilon u)$$
attains its minimum at~$\varepsilon=0$ and therefore
\begin{equation}\label{MiM}
0\le {\mathcal{E}}''(0)=
\int_\Omega |\nabla u|^2 -\sigma u^2 +2\tilde u u^2\,dx.\end{equation}
In particular, the solution is linearly stable,
i.e. the second derivative of the energy is a positive quadratic form.
}\end{remark}
The energy functional is quite useful to capture the
stability of the pure states, such as the ones of the type~$(\tilde u,0)$.
For related approaches, also based on the linearization of
semilinear systems, see e.g.~\cite{Can_Co}.
Also, it is useful to recall that the population~$\tilde u$
cannot beat the resource~$\sigma$, as stated in the following result:

\begin{lemma}\label{TAP}
Consider a bounded function $\sigma:\Omega\to[0,+\infty)$
and a non-negative solution~$\tilde u\in H^1_0(\Omega)$ of~\eqref{eig0}.
Then~$\tilde u(x)\le \|\sigma\|_{L^\infty(\Omega)}$, for any~$x\in\Omega$.
\end{lemma}

\begin{proof} Let~$\Theta:=
\|\sigma\|_{L^\infty(\Omega)}$. We test equation~\eqref{eig0}
against~$v:=\max \{ \tilde u-\Theta,\,0\}$ and we see that
$$ \int_{\Omega} |\nabla v|^2
=\int_\Omega \nabla \tilde u
\cdot \nabla v=\int_\Omega (\sigma-\tilde u)\tilde u v
=\int_{\{\tilde u\ge\Theta\}}
(\sigma-\tilde u)\tilde u(\tilde u-\Theta)
\,.$$
Now observe that, in~$\{\tilde u\ge\Theta\}$,
we have~$\sigma-\tilde u\le \Theta-\tilde u\le0$,
which shows that
$$ \int_{\Omega} |\nabla v|^2\le 0.$$
Accordingly, $v$ vanishes identically and so~$\tilde u\leq\Theta$.
\end{proof}

\begin{corollary}\label{0dj789js}
Consider a bounded function $\sigma:\Omega\to[0,+\infty)$
and a non-negative solution~$\tilde u\in H^1_0(\Omega)$ of~\eqref{eig0}.
Then~$\tilde u$ is continuous inside~$\Omega$.
\end{corollary}

\begin{proof} 
One defines~$\Theta:=\|\sigma\|_{L^\infty(\Omega)}$ and tests equation~\eqref{eig0}
against~$v:=\max \{ \tilde u-\Theta,\,0\}$ to obtain the desired
result (see e.g. \cite{Can_Co}).
\end{proof}


{F}rom now on, we focus on the stability of the system around the stationary point $(\tilde u,0)$, where the distribution of resources $\sigma$ satisfies \eqref{poscond} and $\tilde u\in H^1_0(\Omega)$ is a non-trivial, non-negative solution of \eqref{eig0}. 

The linearization of the system \eqref{syst} at $(\tilde u, 0)$ gives, as a result, the linear operator
\begin{equation}\label{lin_syst}\begin{split}
L_{(\tilde u,0)}(u,v)\,&=\left(
\begin{array}{cc}
\Delta +(\sigma-2\tilde u) & -\tilde u \\
0 & -(-\Delta)^s +(\sigma-\tilde u) 
\end{array}
\right)\left(
\begin{array}{c}
u \\
v
\end{array}
\right)\\
&=
\left(
\begin{array}{c}
\Delta u+(\sigma-2\tilde u)u -\tilde u v
\\ -(-\Delta)^s v +(\sigma-\tilde u)v
\end{array}
\right)
\,,\end{split}
\end{equation}
for any~$(u,v)\in H^1_0(\Omega)\times H^s_0(\Omega)$.
The associated quadratic form, with respect to the duality in~$H^1_0(\Omega)\times H^s_0(\Omega)$,
is
\begin{equation}\label{Q Def}
Q_{(\tilde u,0)}(u,v) =-[u]_{H^1(\R^n)}^2 -[v]_{H^s(\R^n)}^2
+\int_\Omega (\sigma-2\tilde u)u^2 -\tilde u uv
+(\sigma-\tilde u)v^2 \,dx\,,\end{equation}
for any~$(u,v)\in H^1_0(\Omega)\times H^s_0(\Omega)$.
{F}rom the triangular form
of~$L_{(\tilde u,0)}$,
the relevant information is concentrated on the signs of the principal eigenvalues of the pseudodifferential operators on the diagonal
of~\eqref{lin_syst}. 
In this spirit, we first point out that the direction~$(\tilde u,0)$ is
always linearly stable. This is pretty obvious if we think at the biological
model, since~$(\tilde u,0)$ is the stationary configuration
of just one population, and slightly and proportionally
modifying the density of this population
without letting any new population come into the environment should
not drive the system too far from the previous equilibrium.
The formal statement goes as follows:

\begin{lemma}\label{9dv77}
As long as there exists a solution $\tilde u$ for \eqref{eig0}, we have that
$$ Q_{(\tilde u,0)}(\tilde u,0)<0.$$
\end{lemma}
\begin{proof} By testing~\eqref{eig0} against~$\tilde u$,
we obtain that
$$ [\tilde u]_{H^1(\R^n)}^2=\int_\Omega(\sigma-\tilde u)^2\tilde u^2\,dx.$$
As a consequence,
$$ Q_{(\tilde u,0)}(\tilde u,0)=
-[\tilde u]_{H^1(\R^n)}^2
+\int_\Omega (\sigma-2\tilde u)\tilde u^2\,dx
= -\int_\Omega \tilde u^3\,dx.$$
The latter term is strictly negative, thanks to Theorem~\ref{thm:sta_sol}
and so we obtain the desired result.
\end{proof}

We point out that Lemma~\ref{9dv77}
is a particular case of a more general stability result.
Namely, the stationary configuration~$(\tilde u,0)$, which corresponds
to the local population colonizing the whole
of the environment, is also linearly stable with respect
to all the perturbations in which only the the density
of the local species varies
(i.e. the possible source of instability in this setting
may only come from the advent of a nonlocal population).
The formal result goes as follows:

\begin{lemma}
As long as there exists a solution $\tilde u$ for \eqref{eig0}, we have that
$$ Q_{(\tilde u,0)}(u,0)\le0$$
for any~$u\in H^1_0(\Omega)$.
\end{lemma}

This lemma is well-known, due to the variational characterization of the associated eigenvalue problem. We include a proof for the convenience of the reader.
\begin{proof} {F}rom~\eqref{Q Def},
$$ Q_{(\tilde u,0)}(u,0) =-\int_\Omega |\nabla u|^2\,dx
+\int_\Omega (\sigma-2\tilde u)u^2 \,dx,$$
hence the claim follows from~\eqref{MiM}.
\end{proof}

In view of Lemma~\ref{9dv77}, we obtain that
a good way to detect the possible
linear instability of the point $(\tilde u, 0)$ is to rely upon
the perturbations of the form~$(0,v)$, i.e. in the possible
advent of a new population with different diffusive strategy.
The purpose of the next section is therefore to understand
when it is possible to obtain that
$$ Q_{(\tilde u,0)}(0,v_\star)>0,$$
for a suitable choice of~$v_\star\in H^s_0(\Omega)$.


\section{Linear instability}\label{sec:ciccia}

Our aim in this section
is to enlighten the connection between the distribution of resources $\sigma$ and the possible instability of the system, which would suggest some convenience in a nonlocal dispersal strategy of the second species $v$.
For this, we introduce the following notation:

\begin{definition}\label{def:sscat}{\rm
Let $\sigma:\Omega\to[0,+\infty)$ satisfy the
reverse Poincar\'e-Sobolev condition of Definition \ref{def:nontriv}.
Let $\tilde u\ge 0$ be a non-trivial solution of the non-linear equation \eqref{eig0}, provided by Theorem \ref{thm:sta_sol}.
We say that the pair~$(\sigma,\tilde u)$ is mismatched in~$\Omega$ if 
there exists~$x_0\in\Omega$ and~$r>0$ with~$B_r(x_0)\subset\Omega$ and 
\begin{equation}\label{cond2}
\inf_{x\in B_r(x_0)}\big( \sigma(x)-\tilde u(x)\big)
>\frac{1}{C_\sharp(s,B_r(x_0))}\,.
\end{equation}
In this formula,
the constant~$C_\sharp(s,B_r(x_0))$
is the sharp fractional Poincar\'e-Sobolev constant with respect to the ball~$B_r(x_0)$
provided by Theorem \ref{thm_sobineq}.
}\end{definition}

Roughly speaking, condition~\eqref{cond2}
says that the solution~$\tilde u$ is not capable to
exhaust the whole of the resource~$\sigma$ in the whole
of the domain: that is, at least, in the region~$B_r(x_0)$,
the population does not manage to take advantage of all the resource
at its disposal and there is at least a quantity~$C_\sharp(s,\Omega)^{-1}r^{-2s}$
as a leftover. 

In Subsection 3.2 we will see an example of mismatching $(\sigma,\tilde u)$
and it will be clear in that case that the mismatch condition depends
basically on $\sigma$ only.

In our setting, condition~\eqref{cond2} is sufficient to
ensure linear instability, as given by the following result.

\begin{proposition}\label{thm:pos_eig}
If the mismatch
condition in~\eqref{cond2} is satisfied, then
there exists~$v_\star\in H^s_0(\Omega)$
such that~$ Q_{(\tilde u,0)}(0,v_\star)>0$.
\end{proposition}

\begin{proof} By~\eqref{CP}
and~\eqref{cond2}, we know that there exists~$x_0\in\Omega$ and~$r>0$
such that
\begin{equation}\label{BB}
B_r(x_0)\subset\Omega\end{equation}
and
$$ \inf_{x\in B_r(x_0)} \big( \sigma(x)-\tilde u(x)\big)>\frac{1}{ C_\sharp(s,B_r(x_0))} =
\inf_{\substack{\phi\in H^s_0(B_r(x_0)) \\ \phi\not\equiv0}}
\frac{[\phi]_{H^s(\R^n)}^2}{\|\phi\|_{L^{2}(B_r(x_0))}^2}.$$
As a consequence, there exists~$v_\star\in H^s_0(B_r(x_0))$ such that~$v_\star\not\equiv0$ and
\begin{equation}\label{BB2}
\inf_{x\in B_r(x_0)} \big( \sigma(x)-\tilde u(x)\big)>
\frac{[v_\star]_{H^s(\R^n)}^2}{\|v_\star\|_{L^{2}(B_r(x_0))}^2}.\end{equation}
Now notice that~$\|v_\star\|_{L^{2}(B_r(x_0))}=
\|v_\star\|_{L^{2}(\Omega)}$ and~$v_\star$ vanishes
a.e. outside~$\Omega$, thanks to~\eqref{BB}.
This gives that~$v_\star\in H^s_0(\Omega)$.
Moreover, by~\eqref{Q Def}
and~\eqref{BB2},
\begin{align*} 
Q_{(\tilde u,0)}(0,v_\star) & =-[v_\star]_{H^s(\R^n)}^2+\int_{B_r(x_0)} (\sigma-\tilde u)v_\star^2 \,dx\\ 
& >-[v_\star]_{H^s(\R^n)}^2 + \frac{[v_\star]_{H^s(\R^n)}^2}{\|v_\star\|_{L^{2}(B_r(x_0))}^2}
\,\int_{B_r(x_0)} v_\star^2 \,dx=0\,,
\end{align*}
which gives the desired result.
\end{proof}


\begin{remark}{\rm Proposition~\ref{thm:pos_eig} proves the
linear instability of the point $(\tilde u, 0)$ with respect to
perturbation of the type~$(0,v_\star)$ (compare with the theory of Monotone Dynamical Systems in \cite{Smith} or see \cite{Can_Co}). 

Indeed, \[
Q_{(\tilde u,0)}(0,v_\star)=\|v_\star\|_{L^2}^2\lambda(\Omega)\,,
\] 
where $\lambda(\Omega)$ is the principal eigenvalue of the linear pseudodifferential operator $-(-\Delta)^s+(\sigma-\tilde u)$ (see the characterization of the principal eigenvalue by Rayleigh quotient in \cite{Be_Ro_Ro}), thus $L_{(\tilde u, 0)}$ has a negative eigenvalue and a positive one and the stability of a stationary state is determined by the spectrum of the linearization (for this general principle see \cite{Mo}). 

Heuristically, this
can be understood as follows: by formally plugging~$(u,v)=(\tilde u, 0)+
\varepsilon (0,v_\star)+o(\varepsilon)$ into~\eqref{90} we obtain
$$ v_t= -(-\Delta)^s v+(\sigma-(u+v))v
=-\varepsilon(-\Delta)^s v_\star +\varepsilon
(\sigma-\tilde u-\varepsilon v_\star) v_\star+o(\varepsilon).$$
Thus, since~$v_t=\varepsilon\partial_t v_\star+o(\varepsilon)$, we formally
obtain
$$ \partial_t v_\star = 
-(-\Delta)^s v_\star +(\sigma-\tilde u) v_\star+o(1).$$
Hence
$$ \partial_t \|v_\star\|_{L^2(\R^n)}^2
= 2\int_{\R^n} v_\star \partial_t v_\star\,dx
= Q_{(\tilde u,0)}(0,v_\star) +o(1),$$
which is positive by Proposition~\ref{thm:pos_eig}.
 
Therefore,Proposition~\ref{thm:pos_eig} states that the size of the
new population (measured in the $L^2$-norm) has chances to
increase (at least for short times).
}\end{remark}

These type of linearization arguments in the neighborhood
of equilibria that correspond to only one biological species
are widely used in Adaptive Dynamics, see for instance \cite{Diek}, \cite{Ha}, \cite{Hut_Mar_Mi_Vic}
and the references therein.

The rest of this section is devoted to show that the assumptions of Proposition \ref{thm:pos_eig} hold for some $\sigma:\Omega\to\R$.
 
\subsection{Rescaling arguments}\label{subsec:ra}

We propose here a rather simple rescaling argument which gives the existence of a domain $\Omega_\lambda$ and a distribution of resources $\sigma_\lambda$ satisfying the assumptions in Proposition \ref{thm:pos_eig}.
The main drawback of this argument is the fact that the domain $\Omega_\lambda$ changes with the parameter. On the other side, it is immediately evident that the resource $\sigma_\lambda$ leads to instability at $(\tilde u_\lambda,0)$ when it starts being sparse and far from being
homogeneous.

We consider here a smooth function $\sigma:\Omega\to[0,+\infty)$
satisfying the 
reverse Poincar\'e-Sobolev condition
in~\eqref{poscond} (recall Remark~\ref{0oUU}) and 
the corresponding
stationary solution $\tilde u$ given by Theorem \ref{thm:sta_sol}.
We see that, in this case, the population~$\tilde u$
does not exhaust the resource~$\sigma$ in the whole of~$\Omega$.
More precisely, we have:

\begin{lemma}\label{CON90}
Let~$\sigma:\Omega\to[0,+\infty)$
be a 
smooth function 
satisfying the
reverse Poincar\'e-Sobolev condition
in~\eqref{poscond} and
let~$\tilde u$ 
be the corresponding
stationary solution given by Theorem \ref{thm:sta_sol}.

Then there exist~$x_0\in\Omega$, $r>0$ and~$c_0>0$ such that~$B_r(x_0)\subset\Omega$ and
$$ \sigma(x)-u(x)\ge c_0$$
for any~$x\in B_r(x_0)$.
\end{lemma}

\begin{proof} By testing~\eqref{eig0} against~$\tilde u$, we obtain that
$$ 0<\int_\Omega |\nabla \tilde u(x)|^2\,dx=
\int_\Omega (\sigma(x)-\tilde u(x))\tilde u^2(x)\,dx.$$
This implies that there exists~$x_0\in\Omega$ such that~$
\sigma(x_0)-\tilde u(x_0)>0$. The desired result follows
from the continuity of~$\tilde u$ given by Corollary~\ref{0dj789js}.
\end{proof}

In the notation of Lemma~\ref{CON90},
by possibly translating the domain,
we can assume that $x_0=0$, and so
\begin{equation}\label{0cty}{\mbox{$\sigma-\tilde u\ge c_0>0$ in $B_r$.
}}\end{equation}
Then we consider the family of rescaled domains
\[
\Omega_\lambda:=\{\lambda^{-\frac 12}y:\,y\in\Omega\}
\]
and rescaled functions
\[
\sigma_\lambda(x):=\lambda\sigma(\sqrt\lambda x)\,,\quad\forall\,x\in\Omega_\lambda
\]
with $\lambda\ge 1$. Then 
\[
\tilde u_\lambda(x):=\lambda\tilde u(\sqrt\lambda x)\,,\quad\forall\,x\in\Omega_\lambda
\]
is a positive stationary solution for the equation \eqref{eig0} with resource $\sigma_\lambda$, since
\[
\left(\Delta\tilde u_\lambda+(\sigma_\lambda-\tilde u_\lambda)\tilde u_\lambda\right)(x)=\left(\lambda^2\Delta\tilde u+\lambda^2(\sigma-\tilde u)\tilde u\right)(\sqrt\lambda x)=0\,,\quad\forall\,x\in\Omega_\lambda\,.
\]
\begin{proposition}\label{thm:ex_source}
There exists $\Lambda\ge 1$ such that, for every $\lambda\ge\Lambda$, the pair $(\sigma_\lambda,\tilde u_\lambda)$ is
mismatched
in the corresponding domain $\Omega_\lambda$, according to
Definition~\ref{def:sscat}.
\end{proposition}
\begin{proof}
We take~$r_\lambda:=\lambda^{-\frac 12}r$.
By~\eqref{0cty},
\begin{equation}\label{bfd5e}
\begin{split}
& \inf_{|x|<r_\lambda} \big(\sigma_\lambda(x)-\tilde u_\lambda(x)\big)
= \inf_{|x|<\lambda^{-\frac 12}r } \lambda\,
\big(\sigma(\sqrt{\lambda}x)-\tilde u_\lambda(\sqrt{\lambda}x)\big)
\\ &\qquad= \inf_{|y|<r } \lambda\,
\big(\sigma(y)-\tilde u_\lambda(y)\big)\ge c_0\lambda.
\end{split}\end{equation}
On the other hand, by Remark~\ref{rmk:rs},
$$ C_\sharp(s,B_{r_\lambda})=r_\lambda^{2s} C_\sharp(s,B_1)
=\lambda^{-2}r^{2s} C_\sharp(s,B_1).$$
By comparing this with~\eqref{bfd5e}, we conclude that
$$ \inf_{x\in B_{r_\lambda}} \big(\sigma_\lambda(x)-\tilde u_\lambda(x)\big)
\ge c_0\lambda>
\frac{\lambda^{s}}{r^{2s} C_\sharp(s,B_1)}=
\frac{1}{C_\sharp(s,B_{r_\lambda}(x_0))},$$
provided that
\begin{equation*}
\lambda>\left(c_0\,r^{2s}\,C_\sharp(s,B_1)\right)^{-\frac{1}{1-s}}\,.
\qedhere
\end{equation*}
\end{proof}

{F}rom Propositions~\ref{thm:ex_source}
and~\ref{thm:pos_eig}, we obtain that
there exists~$v_{\star,\lambda}\in H^s_0(\Omega_\lambda)$
such that~$ Q_{(\tilde u_\lambda,0)}(0,v_{\star,\lambda})>0$,
as long as $\lambda$ is large enough,
hence~$(\tilde u_\lambda,0)$ is linearly unstable.

This is a first example that shows the validity of Theorem~\ref{MAIN}
(a different one will be constructed in the remaining part of this paper).
It is worth pointing out that the condition that~$\lambda$ is large
translates into the fact that the domain~$\Omega_\lambda$
is small and the resource~$\sigma_\lambda$ is very unevenly
distributed. In some sense, the nonlocal diffusion may allow
the population to take advantage of the small region in which
the resource is abundant, while a less diffusive population
may starve in the portion of the environment with limited
resource.


\subsection{Branching arguments}\label{subsec:cpt}

In this subsection we focus on a particular family of distributions, indeed we assume $B_r(x_0)\subset\Omega$ and
\[
\sigma_\tau(x):=\tau \chi_{B_r(x_0)}(x)=\left\{\begin{array}{lr}
\tau \quad & x\in B_r(x_0)\\
0          & x\notin B_r(x_0)
\end{array}
\right.
\]
We show that there exist $\tau,r>0$ such that the assumptions of Proposition \ref{thm:pos_eig} hold. First of all we have to deal with with Definition \ref{def:nontriv}, which located a branching point for solutions of \eqref{eig0}. For this, for any~$\tau\in\R$, $x_0\in\R^n$, $r>0$, such that~$B_r(x_0)\subset\Omega$,
we introduce the quantity
\begin{equation}\label{RE} 
e(\tau,x_0,r):=\sup_{\substack{u\in H^1_0(\Omega) \\ \|u\|_{L^2(\Omega)}=1}}
\tau\int_{B_r(x_0)}u^2 - \int_\Omega|\nabla u|^2\,.
\end{equation}
We observe that if~$\tau\le 0$ then obviously~$ e(\tau,x_0,r)\le 0$.
Thus we use the
following notation.

\begin{definition}\label{def:ciobar}{\rm We denote 
\[
\underline\tau(x_0,r):=\sup\left\{\tau\in\R\;:\;
e(\tau,x_0,r)\le0
\right\}\,.
\]
}\end{definition}

Now we discuss some basic properties of the quantities that
we have just defined.

\begin{lemma}\label{finite}
The quantity introduced in Definition~\ref{def:ciobar} is finite, namely \[\underline\tau(x_0,r)\in [0,+\infty)\,.\]
\end{lemma}

\begin{proof}
Let~$\phi\in C^\infty_0(B_r)$ with~$\|\phi\|_{L^2(B_r)}=1$,
and let~$u(x):=\phi(x-x_0)$. Then~$\|u\|_{L^2(\Omega)}=
\|u\|_{L^2(B_r(x_0))}=
\|\phi\|_{L^2(B_r)}=1$, and
$$ e(\tau,x_0,r)\ge
\tau\int_{B_r(x_0)}u^2 - \int_{\Omega}|\nabla u|^2
=\tau- \int_{B_r}|\nabla\phi|^2 >0$$
provided that~$\tau>\int_{B_r}|\nabla\phi|^2$.
\end{proof}

\begin{lemma}\label{est}
For any~$\tau_1\le \tau_2$ we have that
$$ e(\tau_2,x_0,r)-e(\tau_1,x_0,r) \in [0,\tau_2-\tau_1].$$
\end{lemma}

\begin{proof} Fix~$\varepsilon>0$. For any~$i\in\{1,2\}$,
there exists~$u_{(i,\varepsilon)}\in H^1_0(\Omega)$,
with $\|u_{(i,\varepsilon)}\|_{L^2(\Omega)}=1$ such that
$$ e(\tau_i,x_0,r)\le \varepsilon+
\tau_i\int_{B_r(x_0)}u_{(i,\varepsilon)}^2 -
\int_\Omega|\nabla u_{(i,\varepsilon)}|^2.$$
Therefore
\begin{eqnarray*}
e(\tau_2,x_0,r)-e(\tau_1,x_0,r) &\ge& 
\tau_2\int_{B_r(x_0)}u_{(1,\varepsilon)}^2 -
\int_\Omega|\nabla u_{(1,\varepsilon)}|^2 -e(\tau_1,x_0,r)
\\ &\ge&
\tau_1\int_{B_r(x_0)}u_{(1,\varepsilon)}^2 -
\int_\Omega|\nabla u_{(1,\varepsilon)}|^2 -e(\tau_1,x_0,r)
\\ &\ge& -\varepsilon,
\end{eqnarray*}
and
\begin{eqnarray*}
e(\tau_1,x_0,r)-e(\tau_2,x_0,r)&\ge&
\tau_1\int_{B_r(x_0)}u_{(2,\varepsilon)}^2 -
\int_\Omega|\nabla u_{(2,\varepsilon)}|^2 -e(\tau_2,x_0,r)
\\ &\ge&
(\tau_1-\tau_2) \int_{B_r(x_0)}u_{(2,\varepsilon)}^2-\varepsilon
\\ &\ge& -(\tau_2-\tau_1)\int_{\Omega}u_{(2,\varepsilon)}^2-\varepsilon
\\ &=& -(\tau_2-\tau_1)-\varepsilon.
\end{eqnarray*}
The desired result now follows by taking~$\varepsilon$ as small as we wish.
\end{proof}

\begin{corollary}\label{cor e tau}
If $\tau\downarrow \underline\tau(x_0,r)$, then~$e(\tau,x_0,r)\to0$.
\end{corollary}

\begin{proof} Suppose not, i.e. there exists a sequence
\begin{equation}\label{PO}\tau_j\ge
\underline\tau(x_0,r)\end{equation}
with~$\tau_j\to \underline\tau(x_0,r)$ as~$j\to+\infty$,
such that
\begin{equation}\label{PO2}
|e(\tau_j,x_0,r)|\ge a,\end{equation}
for some~$a>0$. 
We claim that
\begin{equation}\label{PO2.1}
e(\tau_j,x_0,r)\ge a.\end{equation}
We prove it by contradiction: if not, by~\eqref{PO2},
we would have that~$e(\tau_j,x_0,r)\le-a$. Thus,
we set
$$\tau_a:=\underline\tau(x_0,r)+\frac{a}{2}.$$
We notice that~$\tau_a>\underline\tau(x_0,r)$,
therefore, by Definition~\ref{def:ciobar}, we have that \[e(\tau_a,x_0,r)>0\,.\]
In addition, we have that~$
\tau_a>\tau_j$ if~$j$ is large enough,
thus we make use of
Lemma~\ref{est} and we obtain that, for large~$j$,
$$ 0+a
\le e(\tau_a,x_0,r)-e(\tau_j,x_0,r)\le \tau_a -\tau_j.$$
Taking the limit in~$j$, we conclude that
$$ a\le \tau_a-\underline\tau(x_0,r)=\frac{a}{2}.$$
This is a contradiction and~\eqref{PO2.1} is established.

Also, by Definition~\ref{def:ciobar},
we know that there exists a sequence~$\tilde\tau_j\le \underline\tau(x_0,r)$
with~$\tilde\tau_j\to \underline\tau(x_0,r)$,
such that~$e(\tilde\tau_j,x_0,r)\le0$.
Accordingly, by~\eqref{PO2.1},
\begin{equation}\label{76}
e(\tau_j,x_0,r)-e(\tilde\tau_j,x_0,r)\ge a.
\end{equation}
Notice that~$\tau_j\ge \underline\tau(x_0,r)\ge \tilde\tau_j$ and
$$ \lim_{t\to+\infty} \tau_j-\tilde\tau_j=
\underline\tau(x_0,r)-\underline\tau(x_0,r)=0.$$
Thus, by Lemma~\ref{est}
$$ \lim_{t\to+\infty}
e(\tau_j,x_0,r)-e(\tilde\tau_j,x_0,r) \le \lim_{t\to+\infty} \tau_j-\tilde\tau_j
=0.$$
This is in contradiction with~\eqref{76}
and so the desired result is proved.
\end{proof}
Before stating and proving the main theorem
of this subsection, we investigate the behavior of $\underline\tau(x_0,r)$ under scaling. 
\begin{proposition}\label{prop:stima_ciobar}
Fix $s'\in (0,1)$. There exists a constant $\tau_*:=\tau_*(s',\Omega)$ such that
\[
\underline\tau(x_0,r)\ge r^{-2s'}\tau_*(s',\Omega)
\] 
for every~$x_0\in\Omega$ and~$r>0$ such that~$B_r(x_0)\subset\Omega$.
\end{proposition}
\begin{proof}
We claim that
\begin{equation}\label{CLA}
\int_{B_r(x_0)}u^2\le
c(s',\Omega)\,r^{2s'}\|\nabla u\|^2_{L^2(\Omega)},
\end{equation}
for some constant~$c(s',\Omega)>0$.
Once~\eqref{CLA} is proved, one can finish the proof
of the desired result by arguing as follows. One sets~$\tau_*(s',\Omega):=1/c(s',\Omega)$. Then,
for every~$\tau\le r^{-2s'}\tau_*(s',\Omega)$
(i.e. for every~$\tau\le 1/(c(s',\Omega)\,r^{2s'})$), one has that
\[
\int_\Omega |\nabla u|^2-\tau
\int_{B_r(x_0)} u^2\ge
\int_\Omega |\nabla u|^2-\frac{1}{c(s',\Omega)\,r^{2s'}}
\int_{B_r(x_0)} u^2\ge 0\,,
\]
where the latter inequality is a consequence of the claim \eqref{CLA}.
This gives that $e(\tau,x_0,r)\ge0$ for any~$\tau\le r^{-2s'}\tau_*(s',\Omega)$,
and so, by Definition~\ref{def:ciobar},
we have that~$\underline\tau(x_0,r)\ge r^{-2s'}\tau_*(s',\Omega)$,
thus proving the desired result.

Due to these observations, it only remains to prove~\eqref{CLA}.
To this scope, we observe that, given~$p>2$,
by the H\"older inequality
with exponents $\frac{p}{2}$ and~$\frac{p}{p-2}$, we have
$$ \int_{B_r(x_0)}u^2\le
\left(\omega_n r^n\right)^{\frac{p-2}{p}}\|u\|^2_{L^{p}(\Omega)}.$$
Therefore, the claim in~\eqref{CLA} is established if
we show that there exists~$p>2$ such that
\begin{equation}\label{CLA2}
r^{\frac{(p-2)n}{p}} \|u\|^2_{L^{p}(\Omega)} \le C(s',\Omega,p)\,
r^{2s'}\,\|\nabla u\|^2_{L^2(\Omega)},
\end{equation}
for some~$C(s',\Omega,p)>0$. So, now it only remains to prove~\eqref{CLA2}.
To this goal,
we deal separately\footnote{The case $n\ge 3$ is 
simpler because the Sobolev conjugated exponent $2^*=2n/(n-2)$ is not critical. Indeed, in this case the parameter $s'$ does not play much role.}
with the cases $n=2$ and $n\ge 3$. 

We start with $n\ge 3$. In this case,
we denote by $p:=\frac{2n}{n-2}>2$
the Sobolev conjugate exponent of 2.
Notice that~${\frac{(p-2)n}{p}}=2$ and
the Sobolev inequality (see e.g. formula~(7.26) in~\cite{Gil_Tru})
bounds~$\|u\|^2_{L^{p}(\Omega)}$ with~$C(\Omega)\,\|\nabla u\|^2_{L^2(\Omega)}$,
for some~$C(\Omega)>0$. Hence,
if we denote by~$D_0>0$ the diameter of~$\Omega$, we have that
$$ r^{\frac{(p-2)n}{p}} \|u\|^2_{L^{p}(\Omega)} 
=r^2 \|u\|^2_{L^{p}(\Omega)}\le C_0 r^{2s'} D_0^{2-2s'}
\|\nabla u\|^2_{L^2(\Omega)},$$
and estimate~\eqref{CLA2} follows in this case.

For the case $n=2$, we observe that
$$ \lim_{p\to+\infty} \frac{p-2}{p} = 1> s',$$
so we can choose an even integer~$p=p(s')\in (2,+\infty)$
large enough such that
\begin{equation}\label{ciobar_choice_p}
\frac{p-2}{p}>s'\,.
\end{equation}
Also, the critical Sobolev embedding (see e.g. formula~(7.38)
in~\cite{Gil_Tru}) yields that
\begin{equation}\label{0dhjk}
\int_\Omega {\rm exp}\,\left( \frac{|u(x)|}{c_1 \|\nabla u\|_{L^2(\Omega)}
}\right)^2\,dx\le c_2\,|\Omega|,\end{equation}
for suitable~$c_1$, $c_2>0$. Then, since
$$ e^t=\sum_{k=0}^{+\infty} \frac{t^k}{k!}\ge \frac{t^{p/2}}{(p/2)!},$$
we deduce from~\eqref{0dhjk} that 
$$ \int_\Omega
\left( \frac{|u(x)|}{\|\nabla u\|_{L^2(\Omega)} }\right)^p \,dx\le 
C(\Omega,p),$$
for some~$C(\Omega,p)>0$. Therefore
$$ \|u\|^2_{L^{p}(\Omega)}\le C'(\Omega,p)\,\|\nabla u\|_{L^2(\Omega)}^2 ,$$
for some~$C'(\Omega,p)>0$.
As a consequence, 
if~$D_0>0$ is the diameter of~$\Omega$,
\begin{align*}
r^{\frac{(p-2)n}{p}} \|u\|^2_{L^{p}(\Omega)} & =
r^{2\left(\frac{(p-2)}{p}-s'\right)} \,r^{2s'}\,\|u\|^2_{L^{p}(\Omega)}\\
& \le C'(\Omega,p) \,D_0^{2\left(\frac{(p-2)}{p}-s'\right)}\,r^{2s'}
\|\nabla u\|_{L^2(\Omega)}^2\,.
\end{align*}
This completes the proof of~\eqref{CLA2} when~$n=2$.
\end{proof}
\begin{theorem}\label{thm:percpt} Let~$r$, $\tau>0$.
Consider the family of distributions $\sigma_\tau=\tau\chi_{B_r(x_0)}$ and a corresponding family of stationary solutions $\tilde u_\tau\in H^1_0(\Omega)$, that is
\[
-\Delta\tilde u_\tau=(\sigma_\tau-\tilde u_\tau)\tilde u_\tau\,.
\] 
If $\tau\downarrow \underline\tau(x_0,r)$, then $\tilde u_\tau\to 0$ uniformly.
\end{theorem}
\begin{proof} First of all, we notice that
\begin{equation}\label{bound+}
\tilde u_\tau\le \tau,
\end{equation}
thanks to Lemma~\ref{TAP}.
Now we fix~$\varepsilon\in(0,1)$ and we claim that
\begin{equation}\label{L3}
\| \tilde u_\tau\|_{L^3(\Omega)}\le\varepsilon,
\end{equation}
provided that~$\tau$ is close enough to~$\underline\tau(x_0,r)$.
To establish this, we test the equation against~$\tilde u_\tau$ itself,
and we obtain that
$$\int_\Omega| \nabla \tilde u_\tau|^2=
\int_\Omega (\sigma_\tau-\tilde u_\tau)\tilde u_\tau^2
=\tau \int_{B_r(x_0)} \tilde u_\tau^2 -\int_\Omega \tilde u_\tau^3,$$
which in turn gives
$$ \| \tilde u_\tau\|_{L^3(\Omega)}^3=\int_\Omega \tilde u_\tau^3=
\tau \int_{B_r(x_0)} \tilde u_\tau^2-\int_\Omega |\nabla \tilde u_\tau|^2
\le e(\tau,x_0,r),$$
thanks to~\eqref{RE}.
This and Corollary~\ref{cor e tau}
imply~\eqref{L3}.

Now we set~$g(x):= (\sigma_\tau-\tilde u_\tau)\tilde u_\tau$.
Notice that~$-\Delta u_\tau=g$ in~$\Omega$ and, by~\eqref{bound+}
and Lemma~\ref{finite},
$$ |g|\le(\sigma_\tau+\tilde u_\tau)\,\tilde u_\tau\le
2\tau\,\,\tilde u_\tau\le 2(\underline\tau(x_0,r)+1)\,
\tilde u_\tau\le C_0 \,\tilde u_\tau,$$
for some~$C_0>0$ independent of~$\tau$, as long as~$\tau$
is sufficiently close to~$\underline\tau(x_0,r)$.
In particular, by~\eqref{bound+} and~\eqref{L3},
\begin{equation}\label{SM} \|g\|_{L^{n+3}(\Omega)}\le C_0\,\left( 
\int_\Omega \tilde u_\tau^{n+3}
\right)^{\frac{1}{n+3}} \le C_1
\,\left(
\int_\Omega \tilde u_\tau^{3}
\right)^{\frac{1}{n+3}} \le C_1\varepsilon^{\frac{3}{n+3}},\end{equation}
for some~$C_1>0$.
Moreover, using the H\"older inequality with exponents~$3$ and~$3/2$,
$$ \|\tilde u_\tau\|_{L^2(\Omega)}^2
=\int_\Omega \tilde u_\tau^2
\le  |\Omega|^{\frac13}
\,\left(\int_\Omega \tilde u_\tau^3\right)^{\frac23}=|\Omega|^{\frac13}\,
\|\tilde u_\tau\|_{L^3(\Omega)}^2,$$
therefore, recalling~\eqref{L3} and~\eqref{SM},
$$ \|\tilde u_\tau\|_{L^2(\Omega)}+\|g\|_{L^{n+3}(\Omega)}
\le |\Omega|^{\frac16} \varepsilon+
C_1\varepsilon^{\frac{3}{n+3}}
\le C_2\varepsilon^{\frac{3}{n+3}},$$
for some~$C_2>0$.
We combine this information with Theorem~8.15
of~\cite{Gil_Tru} (used here with~$f:=0$ and~$q:=2(n+3)>n$), thus we obtain that
$$ \|\tilde u_\tau\|_{L^\infty(\Omega)} \le C
\,\big( \|\tilde u_\tau\|_{L^2(\Omega)}+\|g\|_{L^3(\Omega)}\big)\le
C\,C_2\,\varepsilon^{\frac{3}{n+3}},$$
for some~$C>0$,
as long as~$\tau$ is sufficiently close to~$\underline\tau(x_0,r)$,
which is the desired claim.
\end{proof}
\begin{corollary}\label{cor:exist_sigma} 
Fix~$s'\in (s,1)$. Let~$r$, $\tau>0$.
Assume that
\begin{equation} \label{idjjsjsj}
r< \left(\frac{
C_\sharp(s,B_1) \,\tau_*(s',\Omega)}{2}\right)^{\frac{1}{2(s'-s)}},
\end{equation}
where~$C_\sharp(s,B_1)$ is the Poincar\'e-Sobolev constant in~\eqref{CP}
and~$\tau_*(s',\Omega)$ is given by 
Proposition~\ref{prop:stima_ciobar}.

Consider the family of distributions $\sigma_\tau=\tau\chi_{B_r(x_0)}$.
Then there exists~$\tau>\underline\tau(x_0,r)$ such that
both the reverse Poincar\'e-Sobolev condition in~\eqref{poscond}
and the mismatch condition
in~\eqref{cond2}
are satisfied.\end{corollary}
\begin{proof} By taking~$\tau$ large enough, one
can easily fulfill~\eqref{hp_exist}.
This and Lemma~\ref{COMP} guarantee the
reverse Poincar\'e-Sobolev condition in~\eqref{poscond}.

In particular, by Theorem~\ref{thm:sta_sol},
we can consider the solution~$\tilde u_\tau$ corresponding
to the resource~$\sigma_\tau$.

Now we fix
\begin{equation}\label{VE}
\varepsilon\in\left(0,\,\frac{
\tau_*(s',\Omega)}{2\,r^{2(s'-s)} }\right).\end{equation}
Thanks to Theorem \ref{thm:percpt}, we can choose $\tau$ sufficiently 
close to $\underline\tau(x_0,r)$ such that $\|\tilde u_\tau\|_{L^\infty(\Omega)}\le r^{-2s}\varepsilon$.
Therefore, for every $x\in B_r(x_0)$, we have that
\[
\sigma_\tau(x)-\tilde u_\tau(x)\ge\sigma_\tau(x)-r^{-2s}\varepsilon>\underline\tau(x_0,r)-r^{-2s}\varepsilon
.\]
{F}rom this and
Proposition \ref{prop:stima_ciobar}, we have that, for every $x\in B_r(x_0)$,
\[
\sigma_\tau(x)-\tilde u_\tau(x)\ge r^{-2s'}\tau_*(s',\Omega)-r^{-2s}\varepsilon\,.
\]
So, recalling~\eqref{VE},
$$ \inf_{x\in x\in B_r(x_0)}\big(
\sigma_\tau(x)-\tilde u_\tau(x)\big)>\frac{ r^{-2s'}\tau_*(s',\Omega) }{2}.$$
Thus, from Remark~\ref{rmk:rs} and~\eqref{idjjsjsj}, we obtain
\begin{align*}
\frac{1}{C_\sharp(s,B_r(x_0))} & = \frac{1}{r^{2s} C_\sharp(s,B_1)} = \frac{r^{-2s'}\,r^{2(s'-s)}}{
C_\sharp(s,B_1)}\\ 
& <\frac{ r^{-2s'}\, C_\sharp(s,B_1) \,\tau_*(s',\Omega) }{  2\,C_\sharp(s,B_1)}\\
& =\frac{r^{-2s'}\tau_*(s',\Omega)}{2}<\inf_{x\in B_r(x_0)}\left(\sigma_\tau(x)-\tilde u_\tau(x)\right)\,.
\end{align*}
This establishes the mismatch condition
in~\eqref{cond2}.
\end{proof}

{F}rom Proposition~\ref{thm:pos_eig}
and Corollary~\ref{cor:exist_sigma},
it follows that we have constructed another example for which
the equilibrium~$(\tilde u_\tau,0)$ is linearly unstable,
confirming again Theorem~\ref{MAIN}.
Once again, 
this example corresponds to a resource that is unevenly spread
in the environment, and the nonlocal diffusion
may compensate such unbalanced distribution of resource.

As a final observation, we would like to stress that most of
the techniques discussed in this paper are of quite general
nature and can be efficiently exploited in similar problems
with different species and different dispersive properties.

\section{A purely nonlocal phenomenon}\label{sec:cur}

Goal of this section is to rely the stability of stationary points of type $(0,\tilde v)$ with Theorem \ref{NNC} and show how, with our arguments, there is no hope to prove an analogue of Theorem \ref{MAIN} for $(0,\tilde v)$. We include the proof of Theorem~\ref{NNC}
and clarify that it is a purely nonlocal feature.

The linearization of the system \eqref{syst} at $(0,\tilde v)$ gives
\[
L_{(0,\tilde v)}(u,v)=\left(\begin{array}{cc}
\Delta+(\sigma-\tilde v) & 0 \\
-\tilde v & -(-\Delta)^s+(\sigma-2\tilde v)
\end{array}\right)
\left(\begin{array}{c}
u\\
v
\end{array}\right)\,.
\] 
Thus, any instability result would be a consequence of an inequality of type
\[
Q_{(0,\tilde v)}(u_\star,0)=-[u_\star]_{H^1}^2+\int(\sigma-\tilde v)u_\star^2>0\,.
\]

This means that, if we want to run the same argument that we did in Section \ref{sec:ciccia} for $(\tilde u,0)$, then we have to find an analogue for the mismatch condition \eqref{cond2}. Roughly speaking, we need to know that, at least in certain circumstances, the amount of leftovers of the dominant population $\tilde v$ exceeds a given constant, depending on the size of the domain. But, around a stationary point of type $(0,\tilde v)$, the nonlocal population $\tilde v$ tends to exhaust all the available resource $\sigma$ in the domain $\Omega$. This claim is motivated by Theorem \ref{NNC}, because formula~\eqref{CU4} states that the population $u_\varepsilon$
locally fits with any given resource, up to an arbitrarily small
error estimated by~\eqref{CU1}. Of course we are neglecting the Dirichlet boundary condition on $u_\varepsilon$.

We are now left with the proof of Theorem \ref{NNC}. 

\begin{proof}[Proof of Theorem~\ref{NNC}] By Theorem 1.1 in~\cite{Di_Sa_Val},
we know that we can approximate~$\sigma$ by a $s$-harmonic
function in~$B_1$: namely,
we have that there exist~$R_{\varepsilon,\sigma}>1$
and~$u_\varepsilon \in C^k(B_1)\cap C^s(\R^n)$ 
satisfying~\eqref{CU2}, \eqref{CU3} and
\begin{equation}
\label{CU7} \|\sigma-u_\varepsilon\|_{C^k(B_1)}\le\varepsilon.
\end{equation}
Now we define
\begin{equation}
\label{CU8} 
\sigma_\varepsilon:=u_\varepsilon.\end{equation}
In this framework, formula~\eqref{CU1}
follows from~\eqref{CU7} and~\eqref{CU8}.
Moreover, by~\eqref{CU2} and~\eqref{CU8},
$$ (\sigma_\varepsilon(x) - u_\varepsilon(x))\,u_\varepsilon(x)=0=
(-\Delta)^s u_\varepsilon(x),$$
for any~$x\in B_1$, which proves~\eqref{CU5}.\end{proof}

We stress that Theorem~\ref{NNC}
is only due to the nonlocal feature of the equation and
it does not have any local counterpart, as pointed
out by the next result.

\begin{proposition}\label{PCY}
Let~$M>0$.
Let~$\sigma\in C^2(B_1)$ with
\begin{align*}
   &  \sigma(x)\ge M\quad \text{for any }x\in B_{1/16}\hphantom{\setminus B_1}\\
\text{and } \ &   \sigma(x)\le 1\quad \text{for any } x\in B_1\setminus B_{1/10}\,.
\end{align*}
Then, there exists~$M_0>0$ and $\varepsilon>0$ such that, for any~$M\ge M_0$, if
$\sigma_\varepsilon \in C^2(B_1)$ satisfies
\begin{equation}\label{DU1}
\|\sigma-\sigma_\varepsilon\|_{C^2(B_1)}\le\varepsilon
\end{equation}
and~$u_\varepsilon \in C^2(B_1)$ satisfies
\begin{equation}\label{DU5}
-\Delta u_\varepsilon(x) =(\sigma_\varepsilon(x) - u_\varepsilon(x))\,u_\varepsilon(x)
\quad \text{for any }x\in B_1\,,
\end{equation}
then
\begin{equation}\label{DU2}
\|u_\varepsilon-\sigma_\varepsilon\|_{C^2(B_1)}>\varepsilon
\end{equation}
In particular, the local counterpart of Theorem~\ref{NNC} is false.
\end{proposition}

\begin{proof} Suppose by contradiction that for every~$\varepsilon>0$ there exist $\sigma_\varepsilon$
and~$u_\varepsilon$ satisfying not only \eqref{DU1} and \eqref{DU5}, but also
\[
\|u_\varepsilon-\sigma_\varepsilon\|_{C^2(B_1)}\le\varepsilon\,.
\]
{F}rom~\eqref{DU1} and~\eqref{DU2},
we know that
\begin{equation}\label{90gg}
\|u_\varepsilon-\sigma\|_{L^\infty(B_1)}\le
\|u_\varepsilon-\sigma\|_{C^2(B_1)}\le
2\varepsilon.
\end{equation}
As a consequence,
\[ \|u_\varepsilon\|_{L^\infty(B_1)}
\le 2+\|\sigma\|_{C^2(B_1)}\le C_\sigma,\]
for some~$C_\sigma>0$, possibly depending on the fixed resource~$\sigma$.
This, \eqref{DU2} and~\eqref{DU5} give that, in~$B_1$,
$$ |\Delta u_\varepsilon|
\le |\sigma_\varepsilon- u_\varepsilon|\,|u_\varepsilon|\le C_\sigma \,\varepsilon.$$
Thus, the weak Harnack inequality (see e.g. Theorem~8.18
in~\cite{Gil_Tru}) gives that
\begin{equation}\label{WH1}
\|u_\varepsilon\|_{L^1(B_{1/4})}\le C_1\,
\Big( \inf_{B_{1/8}}u_\varepsilon+C_\sigma\,\varepsilon\Big),\end{equation}
for some constant~$C_1>0$.
Now, by~\eqref{90gg} and~\eqref{DU1}, we see that
$u_\varepsilon(x)\ge M-2\varepsilon$ in~$B_{1/16}$ and therefore
\begin{equation}\label{WH2}
\|u_\varepsilon\|_{L^1(B_{1/4})}\ge
\int_{B_{1/16}} u_\varepsilon(x)\,dx \ge C_2\,(M-2\varepsilon),\end{equation}
for some constant~$C_2>0$. Similarly, from~\eqref{90gg} and~\eqref{DU1},
we have that~$u_\varepsilon\le 1+2\varepsilon$ in~$B_1\setminus B_{1/10}$
and therefore
\begin{equation}\label{WH3}
\inf_{B_{1/8}}u_\varepsilon\le 1+2\varepsilon\le 2.
\end{equation}
By inserting~\eqref{WH2} and~\eqref{WH3}
into~\eqref{WH1} we obtain that
$$ M-2\varepsilon\le
C_3\,( 2+C_\sigma\,\varepsilon),$$
for some~$C_3>0$.
Thus, we take~$M\ge M_0:= 3C_3$.
This fixes~$\sigma$ and gives that
$$ C_3\le M-2C_3\le 2\varepsilon
+C_3\,( 2+C_\sigma\,\varepsilon) -2C_3=
(2+C_3 C_\sigma)\,\varepsilon.$$
By taking~$\varepsilon$ small, we obtain a contradiction
and we complete the proof of Proposition~\ref{PCY}.
\end{proof}


\section{Further comments on stability and nontrivial solutions}\label{sec:glob}

Of course, the results presented in this paper do not aim
to exhaust the variety of scenarios offered by the analysis
of local and nonlocal competing species.
In particular, further investigations about
existence and local/global stability of equilibrium solution
are desirable, also with the aim of establishing under which
conditions local and nonlocal strategies are convenient for
the evolution. 

In particular, while we focused here on the local stability
(i.e. whether or not a mutation of strategy turns out to
be persistent for small times), the strategic question
for biological
population in competition for large times is mostly related to
global stability.

The question of attractors for the global dynamics
is related to the regularity properties of the semiflow
and to the associated maximum and comparison principles.
For this reason, though not directly used in this paper,
we present here in detail a general comparison principle
for a single fractional equation (the general case of
systems deserves a separate analysis,
also due to the lack of
cooperativeness between biological species, see e.g. formula~(7)
in~\cite{Boy}, and we plan future further
investigation along the lines of~\cite{HIR1, HIR2}):

\begin{lemma}\label{COMP:LEM}
Let $T>0$ and consider a locally Lipschitz function $f$. Let $v$ and $w$ be bounded and continuous solutions
of
\begin{equation}\label{le091}
\partial_t v+(-\Delta)^s v +f(v)\ge
\partial_t w+(-\Delta)^s w +f(w) \end{equation}
on $\R^n\times(0,T]$, with $v(x,t)\ge w(x,t)$ for any
$x\in\R^n\setminus\Omega$ and any $t\in[0,T]$, and
$v(x,0)\ge w(x,0)$ for any
$x\in\Omega$.

Then $v(x,t)\ge w(x,t)$ for any for any
$x\in\R^n$ and any $t\in [0,T]$.
\end{lemma}

\begin{proof} The proof is of classical flavor,
see e.g. Proposition A.5 in \cite{delaLla_Val}.
By possibly iterating the argument,
it is enough to prove the result up to a small time,
hence, without loss of generality we may assume that 
\begin{equation}\label{t179}
T\le \frac1{4(M+1)},\end{equation}
where $M\ge0$ is the (local) Lipschitz constant of $f$
-- more precisely, we take $M$ such that
\begin{equation}\label{t178}
{\mbox{$f(w(x,t)-\eta)-f(w(x,t))\le M|\eta|$
for any $\eta\in [-1,1]$.}}
\end{equation}
Now we suppose, by contradiction, that the claim were false.
Then, it would exist $(\bar x,\bar t)\in\Omega\times [0,T]$
such that $v(\bar x,\bar t)< w(\bar x,\bar t)$.
We define
\begin{equation}\label{t177} 
\varepsilon :=\min \left\{ \frac{ w(\bar x,\bar t)-
v(\bar x,\bar t)}{4},\; \frac{T}{4}\right\}\end{equation}
and $W:= v-w+\varepsilon t+\varepsilon^2$.
Notice that
\begin{eqnarray*}
&& W(x,0)=v(x,0)-w(x,0) +\varepsilon^2 >0\\ {\mbox{and }}&&
W(\bar x,\bar t) =
v(\bar x,\bar t) -w(\bar x,\bar t) +\varepsilon\bar t+\varepsilon^2
\le v(\bar x,\bar t) -w(\bar x,\bar t) +2\varepsilon <0,
\end{eqnarray*}
thanks to \eqref{t177}.
Hence,
there exists $z_*:=(x_*,t_*)\in\Omega\times[0,\bar t]\subseteq
\Omega\times [0,T]$
such that $W(x,t)>0$ for any $x\in\Omega$ and any $t\in[0,t_*)$,
with $W(x_*,t_*)=0$.

In particular, $W(x,t_*)\ge 0=W(x_*,t_*)$, and so the
integrodifferential definition of the fractional Laplacian gives that
$(-\Delta)^s W(x_*,t_*)\le0$.

Also, $W(x_*,t)\ge0=W(x_*,t_*)$, and thus $\partial_t W(x_*,t_*)\le0$.
In addition, we have that 
$$v(z_*)= W(z_*)+w(z_*)-\varepsilon t_*-\varepsilon^2
= w(z_*)-\eta_*,$$
where $\eta_*:=\varepsilon t_* +\varepsilon^2 \in\left[0,\frac{\varepsilon}{2(M+1)}\right]$,
thanks to \eqref{t179}
and \eqref{t177}.

As a consequence of these observations, and recalling \eqref{le091}
and \eqref{t178}, we find that
\begin{eqnarray*}
0&\ge& \partial_t W(z_*) +(-\Delta)^s W(z_*)\\
&=& \big( \partial_t v+(-\Delta)^s v\big)(z_*)
-\big(\partial_t w-(-\Delta)^s w\big)(z_*) +\varepsilon 
\\ &\ge& f(w(z_*))-f(v(z_*)) +\varepsilon
\\ &=& f(w(z_*))-f(w(z_*)-\eta_*)+\varepsilon\\
&\ge& -M \eta_* +\varepsilon
\\ &>&0,
\end{eqnarray*}
which is a contradicition.
\end{proof}

In addition, we remark that the theory developed in the previous
pages also allows us to investigate the stability
of nonlocal species. For instance, one sees
that large resources allow both 
local and nonlocal populations to stem from the pure equilibria.
More precisely, if~$\sigma$ is larger than the first classical
and fractional eigenvalue, then the
reverse Poincar\'e-Sobolev condition in~\eqref{poscond}
(and its fractional counterpart \eqref{poscond_frac}) are satisfied.

So, in case system~\eqref{90} possesses two unstable
pure equilibria~$(\tilde u,0)$ and~$(0,\tilde v)$,
a positive mixed equilibrium~$ (u_* (x), v_* (x))$
may arise. 

Of course, a detailed analysis of all these circumstances
in a general setting
and a careful check of the nontrivial details involved
by the dynamics associated to the flow
go beyond the scope of this paper, but we refer also to~\cite{Caf_Di_Val}
for a series of examples which carefully 
compare local and nonlocal behaviors of biological populations
in terms of the size of the domain and of the sparseness
of the resources.


\addcontentsline{toc}{chapter}{Bibliography}

\bibliographystyle{plain}

\def\cdprime{$''$}


\end{document}